\documentclass[a4j,11pt]{article}
\usepackage{geometry}                		
\usepackage{amsfonts, amsmath, ascmac,amscd, amssymb,latexsym,amsthm}
\usepackage{graphicx} 
\usepackage{mathrsfs}
\usepackage{color}
\usepackage{overpic}
\usepackage{calc}
\usepackage{here}
\usepackage{bm}

\usepackage{calc} 
\usepackage{accents}

\usepackage{caption}
\renewcommand{\Im}{\operatorname{Im}}
\newtheorem{thm}{Theorem}[section]

\newtheorem{lmm}[thm]{Lemma}

\theoremstyle{definition}
\newtheorem{dfn}[thm]{Definition}

\newtheorem{rem}[thm]{Remark}

\makeatletter
    
    \@addtoreset{equation}{section}
\makeatother
\allowdisplaybreaks
\title{The Borel transform of the WKB solution to the Pearcey system }
\author{Takashi \textsc{Aoki}\footnote{Kindai University, Kowakae 3-4-1, Higashi-Osaka, 
577-8502, Japan.\newline e-mail: \texttt{aoki@math.kindai.ac.jp}},
          Takao \textsc{Suzuki}\footnote{Kindai University, Kowakae 3-4-1, Higashi-Osaka, 
577-8502, Japan.\newline e-mail: \texttt{suzuki@math.kindai.ac.jp}}
~and Shofu \textsc{Uchida}\footnote{Nihon University, Narashinodai 7-24-1, Funabashi,
274-8501, Japan. \newline e-mail: \texttt{uchida.shofu@nihon-u.ac.jp}}}
\date{}

\begin{document}
\maketitle
\begin{abstract}
We study the system of partial differential equations which characterizes the Pearcey integral from the viewpoint of the exact WKB analysis.
It is shown that the Borel transform of the WKB solutions to the system can be written as a linear combination of the branches of an algebraic function of degree $4$. Moreover, some connection formulas for the system are given. 
\end{abstract}
\section{Introduction}
The Pearcey integral is a type of oscillatory integral. This integral is utilized in the analysis of the structure of electromagnetic fields near a cusp of a caustic (\cite{Pea}). It is considered as a two-variable generalization of the Airy integral. \par
In \cite{OK}, particular solutions of the $2$-dimensional Garnier system and its confluent equations are discussed. The system represented by $LH(5)$ in \cite{OK} is obtained by such a confluence. It consists of three second-order
partial differential equations and 
characterizes the linear subspace spanned by the Pearcey integrals. 
On the other hand, another type of system of partial differential equations is investigated in \cite{A2, Hirose} to analyze the Pearcey integral with a large parameter.
This system is essentially equivalent to $LH(5)$ and it includes a partial differential equation of third order.
If one of the independent variables is fixed, this equation is reduced to the differential equation investigated by Berk-Nevins-Roberts \cite{BNR},
where the notion of new Stokes curves is introduced. 
This notion is well understood by using the exact WKB analysis (\cite{AKT, HKT, KT}).
Some basic definitions, such as WKB solutions, turning points, etc.,
for the exact WKB analysis in two variables are introduced in \cite{A2}
and Hirose investigates intensively the connection problem for WKB solutions to the system in \cite{HiroseM, Hirose}.  \par
We introduce a large parameter in $LH(5)$ so that the Pearcey integral with the large parameter represents the solution space.
We also add an extra partial differential equation which represents the weighted homogeneity of the integral with respect to the independent variables and the large parameter. We call this extended system the Pearcey system. 
The aim of this article is to study the relationship between the Pearcey integral and the Borel sum of the WKB solutions to the Pearcey system by using algebraic functions. 
We will show that the Borel transform of the WKB solutions can be written down as linear combinations of the branches of an algebraic function defined by a quartic equation.
As a corollary, we see that the WKB solutions are resurgent functions (\cite{E}).
Using the algebraic equation, we find explicit forms of the connection formulas of Voros type for the WKB solutions (\cite{V}). 
Of course these connection formulas are essentially known by \cite{Hirose}.
Actually, in \cite{HiroseM}, these connection coefficients are given by contour integrals of some $1$-form. 
On the other hand, we obtained explicit values $-1$, $1$ and $0$ for these coefficients.
The difference is due to the choice of normalization of WKB solutions. 
Our method is straightforward and explicit in the sense that we do not use the local reduction theorem at a simple turning point given in \cite{AKT, AKKoT} and our discussion is based on algebraic functions.
We also obtain an explicit equation of the Stokes set of the system of partial differential equations for the Pearcey integral.
Some part of the results of this article is announced in \cite{ASU}. \par
The plan of this article is as follows. In section 2, we introduce the Pearcey system
of partial differential equations as a holonomic system that characterizes the Pearcey
integral with the large parameter. We recall the definition and the construction of
WKB solutions to the Pearcey system in the third section. 
This construction is originally given in \cite{A2, Hirose}, but here we treat the large parameter as an independent variable and this treatment gives a constraint on the normalization of the integral in the definition of the WKB solutions.
In the fourth section, we recall the definition of the turning point set and the Stokes set.
We show that the Stokes set is described by using an algebraic function of degree 6.
In section 5, we give the expression of the Borel transform of the WKB solutions in terms of an algebraic function of degree 4. Using this expression, we can take analytic continuation of the Borel transform of the WKB solutions and hence we obtain connection formulas for the WKB solutions. These are given in the last section as well as some examples.
\section{Pearcey system}
We consider the integral
\begin{align}\label{intPe}
u= \int \exp \left(\eta(z^4+x_2z^2+x_1 z)\right)  dz.
\end{align}
It is called the Pearcey integral (\cite{O}) with a large parameter $\eta$.
Here the integration is taken over an infinite path connecting two
distinct valleys of the integrand. It is easy to see that $u$ satisfies the
following system (\cite[Eq. (0.4)]{OK}):
\begin{align}\label{Pearceysys1}
P_i u=0 \quad(i=1,2,3),
\end{align}
where
\begin{align*}
{P}_1&=4\partial_1 \partial_2+2 \eta  x_2 \partial_1 +\eta^2 x_1,\\
{P}_2&=4 \partial_2^2+  \eta x_1 \partial_1 +2 \eta x_2  \partial_2 +\eta,\\
{P}_3&=\eta \partial_{2}- \partial_{1}^{2}.
\end{align*}
This is  equivalent to
\begin{align}\label{Pearceysys2}
Q_i u=0 \quad(i=1,2),
\end{align}
where
\begin{align*}
{Q}_1&=4\partial_1^3+2 \eta^2  x_2 \partial_1 +\eta^3 x_1,\\
{Q}_2&={P}_3,
\end{align*}
because
\begin{align*}
\begin{split}
&P_1=\eta^{-1}\left(Q_1+4 \partial_1 Q_2\right), \\
&P_2=\eta^{-2} \partial_1 Q_1+\left(4 \eta^{-2} Q_2+8 \eta^{-2} \partial_1^2+2 x_2\right) Q_2, \\
&Q_1=\eta P_1-4 \partial_1 P_3.
\end{split}
\end{align*}
Note that $P_2=\eta^{-1} \partial_1 P_1+2\left(2 \eta^{-1}  \partial_2+x_2\right) P_3$ and that $[P_j , \eta]=0$ ($j = 1, 2, 3$) and $[Q_k , \eta]=0$ ($k = 1, 2$). System \eqref{Pearceysys2} is considered in \cite{A2, Hirose}. We study the system \eqref{Pearceysys2} from a viewpoint slightly different from that of \cite{A2, Hirose}.
We consider $\eta$ as an independent complex variable. Then the system \eqref{Pearceysys1} is subholonomic (\cite{K,Oaku2, Oaku}).
To make a holonomic system for $u$, we look at the weighted homogeneity of the integral  \eqref{intPe}, namely, $u\left(\lambda^{3} x_{1}, \lambda^{2} x_{2}, \lambda^{-4} \eta\right)=\lambda u\left(x_{1}, x_{2}, \eta\right)
$ ($\lambda \neq 0$). Hence $u$ satisfies
$
(3x_1\partial_{1}+ 2x_2\partial_{2} -4\eta\partial_{\eta} -1)u=0,
$
where $\partial_{\eta}=\partial / \partial \eta$. Taking this fact into account, we introduce another partial differential operator
\begin{align*}
{P}_4&=3x_1\partial_{1}+ 2x_2\partial_{2} -4\eta\partial_{\eta} -1
\end{align*}
and consider the system of partial differential equations
\begin{align}\label{Pearceysys}
P_i u=0 \quad(i=1,2,3,4).
\end{align}
Using the notions of $D$-module theory, we may formulate this system as follows.
Let ${D}$  be the Weyl algebra of the variables $(x_1,x_2,\eta)$ and ${ I}$ the left ideal in ${D}$  generated by $P_i$ ($i=1,2,3,4$).
We denote by ${M}$ the left ${D}$-module defined by ${I}$, that is,
\begin{align}\label{Pearsysr}
{ M}={ D}/{ I}.
\end{align}
We call ${M}$ the Pearcey system with a large parameter.

\begin{lmm}\label{gPearRank}
The left ${D}$-module ${ M}$ is a holonomic system of  rank $3$.
\end{lmm}
\begin{proof}
Let $\prec$ denote a monomial order in $D$. For a finite subset {\bf G} of $D$, we set
\begin{align*}
E_{\prec}({\bf G})=\left\{\operatorname{lexp}_{\prec} (G) \mid G \in {\bf G}, G\neq 0\right\},
\end{align*}
where $\operatorname{lexp}_{\prec} (G)$ denotes the leading exponent with respect to the monomial order $\prec$ of $G$ (\cite[Definition 4.2.1]{Oaku}).
Now we consider the monomial order $\eta \succ x_1 \succ x_2$ and set ${\bf{G}}=\{P_i \mid i=1,2,3,4\}$.
Then we can confirm that ${\bf{G}}$ is a Gr\"{o}bner basis of $I$ with respect to the monomial order and
\begin{align*}
E_{\prec}({\bf G})=\{(0,0,0,0,1,1),(0,0,0,0,0,2),(0,0,0,0,2,0),(1,0,0,1,0,0)\}.
\end{align*}
Let $
\varpi: \mathbb{Z}^{6}_{\geq 0}  \rightarrow \mathbb{Z}^{3}_{\geq 0}
$
be the projection defined by
$\varpi(a_1, a_2,a_3,b_1,b_2,b_3)=(b_1,b_2,b_3)$.
Then we have
\begin{align*}
\varpi\left( E_{\prec} (\{ P_1, P_2, P_3, P_4\})\right)=\{(0,1,1),(0,0,2),(0,2,0),(1,0,0)\}
\end{align*}
and hence
\begin{align*}
\#\left(\mathbb{Z}^{3}_{\geq 0} \setminus \operatorname{mono}(\{(0,1,1),(0,0,2),(0,2,0),(1,0,0)\})\right)&=\# \{(0,0,0),(0,1,0),(0,0,1)\}=3,
\end{align*}
where $\operatorname{mono}(F)$ denotes the monoideal generated by $F$ (\cite[Definition 1.1.5]{Oaku}).
This implies that three initial conditions can be given arbitrarily for the analytic solution to \eqref{Pearsysr} at any general point (\cite[Proposition 3]{Oaku2}, \cite[Theorem 5.3.4]{Oaku}). Hence we have the lemma.
\end{proof}
\section{WKB solutions}
We construct WKB solutions of \eqref{Pearceysys1}.
For the unknown function $u$ of \eqref{Pearceysys1}, we set  $S^{(1)}=\partial_1 u / u$ and $S^{(2)}=\partial_2 u / u$. Since the systems \eqref{Pearceysys1} and \eqref{Pearceysys2} are equivalent, we can construct WKB solutions by using \eqref{Pearceysys2}.
Note that WKB solutions to \eqref{Pearceysys2} were constructed in \cite{A2,Hirose}. Hence we have the following system of nonlinear partial differential equations for $S^{(1)}, S^{(2)}$:
\begin{align}\label{S12eq}
\left\{\begin{array}{l}
4\left(S^{(1)}\right)^3+2 \eta^2 x_2 S^{(1)}+\eta^3 x_1+12 S^{(1)} \partial_1 S^{(1)}+4 \partial_1^2 S^{(1)}=0, \\
\eta S^{(2)}-\partial_1 S^{(1)}-\left(S^{(1)}\right)^2=0.
\end{array}\right.
\end{align}
We look for the formal solutions to this system in the form
\begin{align*}
S^{(i)}=S^{(i)}\left(x_1, x_2, \eta\right)=\sum_{j \geq-1} S_j^{(i)}\left(x_1, x_2\right) \eta^{-j} .
\end{align*}
Then the leading term $\zeta=S_{-1}^{(1)}$ of $S^{(1)}$ with respect to the powers of $\eta^{-1}$ should satisfy the cubic equation
\begin{align}\label{zeta}
4\zeta^3+2 x_2 \zeta+x_1=0.
\end{align}
The other terms $S_j^{(1)}(j \geq 0)$ and $S_j^{(2)}(j \geq-1)$ are determined by the following recurrence relations once we take a root $\zeta$ of $\eqref{zeta}$ and set $S_{-1}^{(1)}=\zeta$:
\begin{align} \label{S1rel}
&\left\{\begin{array}{l}
\displaystyle S_0^{(1)} =-\frac{1}{2} \partial_1 \log \left(6 (S_{-1}^{(1)})^2+x_2\right), \\
\displaystyle S_j^{(1)}=-\frac{2}{6 (S_{-1}^{(1)})^2+x_2}\left(\sum_{\substack{j_1+j_2+j_3=j-2 \\
-1 \leq j_1, j_2, j_3<j}} S_{j_1}^{(1)} S_{j_2}^{(1)} S_{j_3}^{(1)}\right.\\
\hspace{50mm}\displaystyle\left. +3 \sum_{\substack{j_1+j_2=j-2 \\
-1 \leq j_1, j_2<j}} S_{j_1}^{(1)} \partial_1 S_{j_2}^{(1)}+\partial_1^2 S_{j-2}^{(1)}\right)\quad (j \geq 1),
\\ 
\end{array}\right.\\ \label{S2rel}
&\left\{\begin{array}{l}
S_{-1}^{(2)}=\displaystyle\left(S_{-1}^{(1)}\right)^{2},\\
S_{j}^{(2)}=\displaystyle\sum_{m=0}^{j+1} S_{m-1}^{(1)}S_{j-m}^{(1)}+\partial_1 S_{j-1}^{(1)} \quad(j\geq0).
\end{array}\right.
\end{align}
We have three roots of \eqref{zeta} for general $x_1, x_2$. 
Hence, if $x_{1}\neq 0$, the roots are labeled as $\zeta_{\ell}$ ($\ell=1,2,3$) under suitably taken
branch cut hypersurfaces. For example, for $x_{2}=0$, we take the negative real axis $(-\infty,0]$ 
in the $x_{1}$-plane as the branch cut, namely, $-\pi<\arg x_{1}\le \pi$, 
and set $\zeta_{3}=-x_{1}^{1/3}/4^{1/3}, \,\zeta_{1}=\zeta_{3}e^{2\pi i/3},\, \zeta_{2}=\zeta_{3}e^{4\pi i/3}$.
For $x_{2}\neq 0$, we take the rays in the $x_{1}$-plane 
starting at $\pm\sqrt{8/27}\, i\, x_{2}^{3/2}$ (the branch points in $x_{1}$-space) 
with the real negative direction as the branch cuts. Here, when the ray starting
from one of the branch points hits another branch point, the ray should be deformed infinitesimally  
at the hitting point so
that it detours the point upper (or lower)  side. Using thus obtained cut hypersurfaces, we can
label the roots consistently outside the cut hypersurfaces. 
We may deform the directions of the rays if necessary. \par
We set $S_{-1}^{(1), \ell}=\zeta_{\ell}$. Accordingly, we have three formal solutions $\left(S^{(1), \ell}, S^{(2), \ell}\right)$ to \eqref{S12eq} ($\ell=1,2,3$). Let $\omega^{(\ell)}$ denote the 1-form $S^{(1), \ell} d x_1+S^{(2), \ell} d x_2$. Then this 1-form is closed (\cite{A2, Hirose}). A formal solution of the form
\begin{align}\label{WKBsolh}
\psi^{(\ell)}=\eta^{-1 / 2} \exp \left(\int_{\left(a_1, a_2\right)}^{\left(x_1, x_2\right)} \omega^{(\ell)}\right)
\end{align}
is called a WKB solution to \eqref{Pearceysys2} ($\ell=1,2,3$). Here $\left(a_1, a_2\right)$ is any fixed general point. Since \eqref{Pearceysys2} is equivalent to \eqref{Pearceysys1}, $\psi^{(\ell)}$ may be called WKB solutions to \eqref{Pearceysys1}.
\par
In \cite{A2, Hirose}, $\eta$ is considered to be a parameter. Our holonomic system $M$ contains $\eta$ as an independent variable and $P_4 \psi^{(\ell)}$ does not vanish in general because $\psi^{(\ell)}$ defined by \eqref{WKBsolh} is not weighted homogeneous. This means that $\psi^{(\ell)}$ is not a formal solution to $M$. Of course, $\eta$ plays a special role, namely, an asymptotic parameter. Basically we regard it as an independent variable, however, sometimes we treat it as a parameter according to the situation. To find formal solutions to $M$ of the form
\begin{align}\label{WKBsolf}
\psi_{\ell}=\eta^{-1 / 2} \exp \left(\int \omega^{(\ell)}\right),
\end{align}
we have to take the primitives $\int \omega^{(\ell)}$ so that $\psi_{\ell}$ satisfies $P_4 \psi_{\ell}=0$.
\begin{lmm}\label{lmm3.1}
Let $\omega_j^{(\ell)}$ denote the coefficient of $\eta^{-j}$ in $\omega^{(\ell)}$ $(j=-1,0,1,2, \ldots)$. Then we have
\begin{align} \label{omegaj1}
\int \omega_0^{(\ell)} &=-\frac{1}{2} \log \left(6\left(S_{-1}^{(1), \ell}\right)^2+x_2\right)+C, \\ \label{omegaj2}
\int \omega_j^{(\ell)} &=-\frac{1}{4 j}\left(3 x_1 S_j^{(1), \ell}+2 x_2 S_j^{(2), \ell}\right) \quad(j \geq-1, j \neq 0).
\end{align}
Here $C$ is an arbitrary constant and we set $\int \omega^{(\ell)}=\sum_{j \geq-1} \eta^{-j} \int \omega_j^{(\ell)}$.
\end{lmm}
\begin{proof}
From $P_4 \psi_{\ell}=0$,  we have
\[3 x_1 S^{(1),\ell} + 2 x_2 S^{(2),\ell} + 1 - 4 \eta \partial_{\eta} \int
\omega^{(\ell)} = 0.\]
Substituting $S^{(i),\ell}=\sum_{j \geq-1} \eta^{-j} S_j^{(i),\ell}$ ($i=1,2$) and $ \int \omega^{(\ell)}=\sum_{j \geq-1} \eta^{-j} \int \omega_j^{(\ell)}$, we have
\begin{align}\label{lmm3.1_temp}
\sum_{j \geq-1} \eta^{-j}\left(3 x_1 S_j^{(1),\ell} + 2 x_2 S_j^{(2),\ell} \right)+1+4 \sum_{j \geq-1,j\neq0} j \eta^{-j}\int \omega_j^{(\ell)}=0.
\end{align}
By comparing the coefficients of each power of $\eta$ in the equation \eqref{lmm3.1_temp}, we obtain the second equation \eqref{omegaj2} of Lemma \ref{lmm3.1}. The first equation  \eqref{omegaj1} can be derived easily by using the explicit forms of $S_0^{(k), \ell}$ and \eqref{zeta}.
\end{proof}

Therefore we call \eqref{WKBsolf} the WKB solutions to $M$. Here the primitives $\int \omega^{(\ell)}$ are given by Lemma $3.1$ with $C=0$. The explicit forms of $\psi_{\ell}$ are given by
\begin{align}\label{WKBsolp}
\psi_{\ell}=\frac{\exp \left(\displaystyle\frac{\eta}{4}\left(3 x_1 S_{-1}^{(1), \ell}+2 x_2 S_{-1}^{(2), \ell}\right)\right)}{\left(\eta \left(6\left(S_{-1}^{(1), \ell}\right)^2+x_2\right)\right)^{1/2}} \exp \left(-\sum_{j \geq 1} \frac{\eta^{-j}}{4 j}\left(3 x_1 S_j^{(1), \ell}+2 x_2 S_j^{(2), \ell}\right)\right).
\end{align}

\section{Turning point set and Stokes set}
We give the definition of turning points of $M$. 
\begin{dfn}
Let $\ell, k \in\{1,2,3\}$ satisfying $\ell \neq k$ and $\tau_1\neq 0$. A point $\tau=\left(\tau_1, \tau_2\right) \in \mathbb{C}^2$ is called a turning point of $M$ of type $(\ell, k)$ if
\begin{align*}
\omega_{-1}^{(\ell)}\left(\tau\right)=\omega_{-1}^{(k)}\left(\tau\right)
\end{align*}
holds. Here the labeling of the 1-form $\omega^{(\ell)}$ (or $S_{-1}^{(j), \ell} ; j=1,2$) is introduced in the previous section.
The set of all turning points of some type (resp. of type $(\ell, k)$) is denoted by $T'$ (resp. $T'_{\ell k}$).
We call $T=T' \cup \{(0,0)\}$ the turning point set of $M$.
\end{dfn}
Since all the roots of \eqref{zeta} merge at the origin, the corresponding 1-forms $\omega$
all coincide at the origin. Therefore, we include the origin in the turning point set.
On the other hand, when $x_1=0$ and $x_2\neq 0$, the 1-form $\omega$ defined for any two roots of $4\zeta^3+2x_2\zeta=0$ are distinct, hence we do not include $\{(x_1, x_2) \in \mathbb{C}^2\mid x_1=0, x_2\neq 0 \}$ in the turning point set.
This definition is the same as that of turning points of \eqref{Pearceysys2} which is established in \cite{A2, Hirose}. As is shown in these articles, $T$ coincides with the zero points of the discriminant of \eqref{zeta} with respect to $\zeta$:
\begin{align*}
T=\left\{\left(x_1, x_2\right) \in \mathbb{C}^2 \mid 27 x_1^2+8 x_2^3=0\right\} .
\end{align*}
\begin{dfn}
Let $\tau=\left(\tau_1, \tau_2\right) \neq(0,0)$ be a turning point of type $(\ell, k)$. A Stokes surface passing through $\tau$ is the set of all $x=\left(x_1, x_2\right) \in \mathbb{C}^2$ with $x_1 \neq 0$ satisfying
\begin{align*}
\operatorname{Im} \int_\tau^x\left(\omega_{-1}^{(\ell)}-\omega_{-1}^{(k)}\right)=0.
\end{align*}
This surface is denoted by $\mathcal{S}_\tau$. We say that $\mathcal{S}_\tau$ has the type $(\ell, k)$. Moreover, we set
\begin{align*}
\mathcal{S}^{\circ}=\bigcup_{ \tau \in T \setminus \{(0,0)\} } \mathcal{S}_\tau 
\end{align*}
and denote the closure of  $\mathcal{S}^{\circ}$ by $\mathcal{S}$. We call $\mathcal{S}$ the Stokes set of $M$.\\
\end{dfn}
The above union has redundancy. 
Let $\tau,\tau'\in T\setminus\{(0,0)\}$ be turning points of the same type $(\ell,k)$.
Let $x=(x_1,x_2)\in\mathbb C^2\setminus T$ with $x_1\neq0$. 
Since $T\setminus\{(0,0)\}$ is connected, we can choose a path $\gamma$ in $T\setminus\{(0,0)\}$ from $\tau'$ to $\tau$.
We label the pair coalescing along $\gamma$ by $(\ell,k)$.
The paths of analytic continuation from $\tau$ to $x$ and from $\tau'$ to $x$ 
are chosen so that both continuations yield the same branches $\zeta_\ell$ and $\zeta_k$ at $x$.
Note that $\omega_{-1}^{(\ell)}$ and $\omega_{-1}^{(k)}$ on $\gamma$ are
understood as limits from $\mathbb C^2\setminus T$ on the chosen branches.
Then
\[\operatorname{Im} \int_{\tau'}^x \left(\omega_{-1}^{(\ell)}-\omega_{-1}^{(k)}\right)=\operatorname{Im} \int_{\gamma} \left(\omega_{-1}^{(\ell)}-\omega_{-1}^{(k)}\right)+\operatorname{Im} \int_{\tau}^{x} \left(\omega_{-1}^{(\ell)}-\omega_{-1}^{(k)}\right)=\operatorname{Im}\int_\tau^x
   \left(\omega_{-1}^{(\ell)}-\omega_{-1}^{(k)}\right)\]
holds.
Hence we have
$\mathcal{S}_{\tau^{\prime}}=\mathcal{S}_\tau$. 
We can express $\mathcal{S}$ as the closure of 
\begin{align*}
\bigcup_{\ell < k} \left\{
x = (x_1, x_2) \in \mathbb{C}^2 \ \middle|\
\begin{tabular}{l}
$x_1 \neq 0,\displaystyle \lim_{\substack{ \varepsilon \to 0\\ \varepsilon \in T'_{\ell k}}}  \operatorname{Im} \int_{\varepsilon}^{x} \left( \omega_{-1}^{(\ell)} - \omega_{-1}^{(k)} \right) = 0$,
where the path of \\
integration does not
intersect the set $\{x_1 = 0\}$
\end{tabular}
\right\}.
\end{align*}
A connected component of ${\mathbb{C}}^2 \setminus \mathcal{S}$ is called a Stokes region.
\begin{thm}
Let $F$ be the algebraic function defined by
\begin{align}\label{Feq}
16F^6+32 x_2(27x_1^2-x_2^3)F^4+16x_2^2(27x_1^2-x_2^3)^2 F^2+x_1^2(27x_1^2+8x_2^3)^3=0.
\end{align}
The Stokes set $\mathcal{S}$ of the Pearcey system ${M}$ is described as
\begin{align*}
\mathcal{S}=\{ (x_1,x_2) \in \mathbb{C}^2 \mid  \operatorname{ Im} F=0\}.
\end{align*}
Hence it is a semialgebraic set as a subset of $\mathbb{C}^2\simeq\mathbb{R}^4$. 
\end{thm}
\begin{rem}
It has been shown by \cite{Hirose2, L} that $\mathcal{S}$ is a semialgebraic set under general assumptions.
\end{rem}
\begin{proof}
Using the primitive $\int \omega_{-1}^{(\ell)}$ given by \eqref{omegaj2} and the definition of $S_{-1}^{(j),\ell}$, we can write
\begin{align}\label{Feqomega}
\begin{split}
\int_{\tau}^x\left(\omega_{-1}^{(\ell)}-\omega_{-1}^{(k)}\right)&=\frac{1}{4}(3x_1S_{-1}^{(1),\ell}+2x_2S_{-1}^{(2),\ell})-\frac{1}{4}(3x_1S_{-1}^{(1),k}+2x_2S_{-1}^{(2),k})\\
&=\frac{1}{4}(3x_1(S_{-1}^{(1),\ell}-S_{-1}^{(1),k})+2x_2((S_{-1}^{(1),\ell})^2-(S_{-1}^{(1),k})^2))\\
&=\frac{1}{4}(S_{-1}^{(1),\ell}-S_{-1}^{(1),k})(3x_1+2x_2(S_{-1}^{(1),\ell}+S_{-1}^{(1),k})).
\end{split}
\end{align}
Here $\tau$ denotes a turning point of type ($\ell,k$). Let $F$ be an indeterminate. Eliminating $\zeta_\ell$ and $\zeta_k$ from the relations
\begin{align}\label{Feqz1}
&4 \zeta_\ell^3 +2x_2\zeta_\ell+x_1=0,\\ \label{Feqz2}
&4 \zeta_k^3 +2x_2\zeta_k+x_1=0,\\ \label{Feqz3}
&F=(\zeta_\ell-\zeta_k)(3x_1+2x_2(\zeta_\ell+\zeta_k)),
\end{align}
we have \eqref{Feq}.
To be more specific, we firstly compute the resultant of $4 \zeta_\ell^3 +2x_2\zeta_\ell+x_1$ and $(\zeta_\ell-\zeta_k)(3x_1+2x_2(\zeta_\ell+\zeta_k))-F$ with respect to $\zeta_\ell$ and denote it by $\Delta$. Next we compute the resultant of  $4 \zeta_k^3 +2x_2\zeta_k+x_1$ and $\Delta$ with respect
to $\zeta_k$. Since $F\not \equiv 0$, we have $\mathcal{S}\subset\{ (x_1,x_2) \in \mathbb{C}^2 \mid  \operatorname{ Im} F=0\}$. \par
Inversely, we can show 
$\mathcal{S}\supset\{ (x_1,x_2) \in \mathbb{C}^2 \mid  \operatorname{ Im} F=0\}$
by going back the way we came.

\end{proof}
\section{Borel transform of WKB solutions}
\noindent
The WKB solution $\psi_\ell$ of the Pearcey system $M$ (see \eqref{WKBsolp}) can be expanded in the form:
\begin{align*}
\psi_\ell=\exp  \left(\eta\varpi_\ell \right)\sum_{j=0}^{\infty} \eta^{-j-\frac{1}{2}} f_{j,\ell}\left(x_{1}, x_{2}\right) \qquad (\ell=1,2,3),
\end{align*}
where
$
\varpi_\ell= \frac{1}{4}\left(3 x_{1} S_{-1}^{(1),\ell}+2 x_{2} S_{-1}^{(2),\ell}\right)
$
and
$f_{0,\ell}\left(x_{1}, x_{2}\right)=\left(6 (S^{(1),\ell}_{-1})^{2}+x_{2}\right)^{-\frac{1}{2}}$. The Borel transform of $\psi_\ell$ is denoted as $\psi_{\ell,B}$:
\begin{align*}
\psi_{\ell,B}=\sum_{j=0}^{\infty} \frac{f_{j,\ell}\left(x_{1}, x_{2}\right)}{\Gamma(j+1 / 2)}\left(y+\varpi_\ell\right)^{j-1 / 2}.
\end{align*}
For the differential operator 
$
P=\sum_{i, j, k, m} c_{i, j}\left(x_{1}, x_{2}\right) \eta^{k} \partial_{1}^{i} \partial_{2}^{j} \partial_{\eta}^{m} \in D
$,
we set
\begin{align*}
P_B=\sum_{i, j, k, m} c_{i, j}\left(x_{1}, x_{2}\right) \partial_{y}^{k} \partial_{1}^{i} \partial_{2}^{j}(-y)^{m}
\end{align*}
and we call $P_B$ the Borel transform of $P$.
Let $D_B$ denote the Weyl algebra
of the variable $(x_1, x_2, y)$. Then $P_B$ belongs to $D_B$.
The Borel transform of $P_i$ ($i=1,2,3,4$) are given as follows:
\begin{align}\label{reduPeB}
\begin{split}
P_{1,B}&=4\partial_1 \partial_2+2 x_2 \partial_1\partial_y +x_1\partial_y^2,\\
P_{2,B}&= 4 \partial_2^2+x_1 \partial_1 \partial_y +2x_2  \partial_2 \partial_y+\partial_y,\\
P_{3,B}&=\partial_2\partial_y -\partial_1^2,\\
P_{4,B}&=3 x_1\partial_1+2 x_2\partial_2+4 y\partial_y+3.
\end{split}
\end{align}
Let $I_B$ be the left ideal of $D_B$ generated by $P_{k,B}$ ($k = 1, 2, 3, 4$). By the
definition, $P_{j,B}\psi_{\ell,B}= 0$ holds for $j = 1, 2, 3, 4; \ell = 1, 2, 3$. The following lemma can be proved in a similar manner to Lemma \ref{gPearRank}.
\begin{lmm}\label{gBPearRank}
Let ${M_B}$ denote the left ${D_B}$-module defined by ${I_B}$:
\begin{align}\label{Pearsysrb}
{M_B}={D_B}/{I_B}.
\end{align}
Then, the left ${D_B}$-module ${M_B}$ is a holonomic system of  rank $3$.
\end{lmm}
\noindent
Since $\psi_{\ell,B}$ are linearly independent, $(\psi_{1,B},\psi_{2,B},\psi_{3,B})$ forms a basis of the analytic solution space of $M_B$. In other words, $M_B$ characterizes the subspace of analytic functions spanned by $\psi_{\ell,B}$ ($\ell = 1, 2, 3$).

The Borel transform $\psi_{\ell,B}$ has a singularity at $u_\ell:=-\varpi_\ell$ and $u = u_\ell$ $(\ell=1,2,3)$ satisfies
\begin{align}\label{ueqq}
256u^3-128x_2^2u^2+16x_2(9x_1^2+x_2^3)u-x_1^2(27x_1^2+4x_2^3)=0.
\end{align}
We obtain this by eliminating $\zeta$ from $u=-\frac{1}{4}\left(3x_1\zeta+2x_2\zeta^2\right)$ and $4\zeta^3+2x_2 \zeta+x_1=0$.
The Stokes set $\mathcal{S}$ is also expressed by using the roots of \eqref{ueqq}:
\[
\mathcal{S}= \overline{ \bigcup_{j \neq k}\left\{(x_1,x_2) \in \mathbb{C}^2 \mid {\rm Im} (u_j-u_k)=0\right\}},
\]
where the overline denotes closure in $\mathbb{C}^2$.
Note that, by the definition of $u_\ell$ and $\varpi_\ell$,  we have $u_\ell= -\int_0^x \omega_{-1}^{(\ell)}$ since \eqref{omegaj2} holds.
Using the method of \cite[Theorem 5.3.3]{Oaku},
we can show that the singular locus of $M_B$ is equal to
\begin{align*}
\begin{split}
\{(x_1,x_2,y) \in \mathbb{C}^3\mid 256y^3-128x_2^2y^2+16x_2(9x_1^2+x_2^3)y-x_1^2(27x_1^2+4x_2^3)=0\}.
\end{split}
\end{align*}
It coincides with the zeros of the discriminant of $z^{4}+x_{2} z^{2}+x_{1} z+y$ with respect to $z$.
\par
Since $P_{4,B}\psi_{\ell,B}=0$, the Borel transform of the WKB solution has the weighted homogeneity 
\begin{align*}
\psi_{\ell, B}\left(\lambda^3 x_1, \lambda^2 x_2, \lambda^4 y\right)=\lambda^{-3} \psi_{\ell, B}\left(x_1, x_2, y\right) 
\end{align*}
for $\lambda \neq 0$. If $x_1\neq0$, we have 
\begin{align*}
\psi_{\ell, B}\left(1, \frac{x_2}{x_1^{2 / 3}}, \frac{y}{x_1^{4 / 3}}\right)=x_1 \psi_{\ell, B}\left(x_1, x_2, y\right).
\end{align*}
Hence we introduce new variables $s$ and $t$ by setting
\begin{align*}
\displaystyle(s,t) =\left (\frac{y}{x_1^{4/3}}, \frac{x_2}{x_1^{2/3}}\right).
\end{align*}
Let $\displaystyle p_\ell=\frac{3}{2^{8/3}}  e^{\frac{2 \pi i}{3} \ell}$ ($\ell=1,2,3$).
Then \(p= p_\ell \) satisfies the following cubic equation
\begin{align*}
256p^3 - 27 = 0.
\end{align*}
This equation corresponds to the equation \eqref{ueqq} when restricted to \( x_1 = 1\) and \(x_2 = 0 \).
If necessary, we may relabel the $\zeta_\ell$ such that the restriction of $u_\ell$ to $x_1 = 1$ and $x_2 = 0$ can be regarded as $p_\ell$.
Expanding $x_1 \psi_{\ell,B} \Big|_{t=0}$ at $s=p_\ell$ yields
\begin{align}\label{WKBBorel}
\begin{split}
x_1 \psi_{\ell,B} \Big|_{t=0}
&=-\frac{2^{1/6}}{\sqrt{3 \pi}}e^{-\frac{2 \pi i}{3}\ell}\left(s-p_\ell \right)^{-1/2}\\
&\qquad \times \left(1-\frac{7}{9\cdot  2^{1/3}}e^{-\frac{2 \pi i}{3}\ell} \left(s-p_\ell \right)+\frac{385}{486\cdot  2^{2/3}}e^{\frac{2 \pi i}{3}\ell}\left(s-p_\ell \right)^{2}+O\left(\left(s-p_\ell \right)^{3}\right)\right).
\end{split}
\end{align}
Here the argument of $s-p_\ell$ is chosen as
\begin{align*}
-\pi< \arg\left(s-p_\ell\right)\leq\pi
\end{align*}
for $s\in \mathbb{C}\setminus \{p_{\ell}\}$ ($\ell=1,2,3$).
Next we take the argument of $p_\ell-s$ as
\begin{align}\label{p-s}
-\pi< \arg\left(p_\ell-s\right)\leq\pi
\end{align}
for $s\in \mathbb{C}\setminus \{p_{\ell}\}$ ($\ell=1,2,3$).
Then we have
\begin{align}\label{p1s}
(p_1-s)^{-1/2}&=-i(s-p_1)^{-1/2}&&(s \in \{ e^{2\pi i/3} \sigma \mid  0<\sigma<p_3\}),\\ \label{p2s}
(p_2-s)^{-1/2}&=i(s-p_2)^{-1/2}&&(s \in \{ e^{4\pi i/3}\sigma \mid  0<\sigma<p_3\}),\\ \label{p3s}
(p_3-s)^{-1/2}&=i(s-p_3)^{-1/2}&& (0<s<p_3).
\end{align}
We go back to the Pearcey integral \eqref{intPe}. Changing the integration variable by $y=-(z^{4}+x_{2} z^{2}+x_{1} z)$, we can rewrite it to the form of the Laplace integral:
\begin{align}\label{uBint}
\int \exp (-\eta y) g\left(x_{1}, x_{2}, y\right) d y.
\end{align}
Here we set
\begin{align*}
g\left(x_{1}, x_{2}, y\right)=-\left.\frac{1}{4 z^{3}+2 x_{2}z+x_{1}}\right|_{z=z\left(x_{1}, x_{2}, y)\right.}
\end{align*}
for a root $z = z(x_1 ,x_2 ,y)$ of the quartic equation $z^4 + x_2 z^2 + x_1 z + y = 0$ and
here we take the path of integration as the image of that of \eqref{intPe} by $z\mapsto y$. We
observe that \eqref{uBint} resembles the Laplace integral that defines the Borel sum of the WKB solutions and hence we may expect some relations between $g$ and $\psi_{\ell,B}$ (see \cite{A2}).  
Actually it is shown in \cite{HKT, TakeiRIMS} that when $x_2$ is fixed,  \eqref{uBint} corresponds to the Borel sum of a WKB solution of $Q_1u=0$.

\begin{lmm}
The function $g$ is a root of the quartic equation
\begin{align}\label{geq}
\begin{split}
(4 x_1^2 x_2(36  y-x_2^2)+16 y(x_2^2-4 y)^2-27 x_1^4)g^4+2(-8x_2 y+2 x_2^3+9 x_1^2)g^2 -8x_1g+1=0
\end{split}
\end{align}
and it is a solution to the holonomic system $M_B$ 
\end{lmm}
\begin{proof}
Differentiating \eqref{geq} by $x_{1}, x_{2}, y$ several times, we can express
all the derivatives of $g$ in terms of $g$. For example, we have
\begin{align*}
\partial_{1}g&=\frac{g \left(g^3 x_1 \left(-72 x_2 y+2 x_2^3+27 x_1^2\right)-9 g
   x_1+2\right)}{g^3 \left(4 x_1^2 \left(36 x_2 y-x_2^3\right)+16 y
   \left(x_2^2-4 y\right){}^2-27 x_1^4\right)+g \left(-8 x_2 y+2
   x_2^3+9 x_1^2\right)-2 x_1},\\
\partial_{2}g&=\frac{g^2 \left(g^2 \left(x_1^2 \left(3 x_2^2-36 y\right)+16 x_2 y
   \left(4 y-x_2^2\right)\right)-3 x_2^2+4 y\right)}{g^3 \left(4 x_1^2
   \left(36 x_2 y-x_2^3\right)+16 y \left(x_2^2-4 y\right){}^2-27
   x_1^4\right)+g \left(-8 x_2 y+2 x_2^3+9 x_1^2\right)-2 x_1},\\
\partial_{y}g&=\frac{4 g^2 x_2-4 g^4 \left(-16 x_2^2 y+x_2^4+9 x_1^2 x_2+48
   y^2\right)}{g^3 \left(4 x_1^2 \left(36 x_2 y-x_2^3\right)+16 y
   \left(x_2^2-4 y\right){}^2-27 x_1^4\right)+g \left(-8 x_2 y+2
   x_2^3+9 x_1^2\right)-2 x_1},
 \end{align*}
etc. \\
 \noindent $P_{4,B}g=0$ is easily seen from the above equations.
To see $P_{j,B}g=0$ ($j=1,2,3$), we express the derivatives of $g$ appearing in
$P_{j,B}g$ in terms of $g$ and put all the terms over a common denominator.
Dividing the numerator of thus expressed $P_{j,B}g$ by the left-hand side of \eqref{geq} 
as a polynomial of $g$, we see that the remainder vanishes, namely $P_{j,B}g=0$.
Although all the procedures are done by using only elementary calculus,
explicit forms of equations in these computations are quite long and messy,
which are omitted here.
\end{proof}
For general $\left(x_1, x_2, y\right)$, there are four roots $g_k$ ($k=1,2,3,4$) of the quartic equation, which satisfy $g_1+g_2+g_3+g_4=0$. 
We observe that the algebraic function $g$ defined by \eqref{geq} has the same weighted homogeneity as that of $\psi_{\ell, B}$ with respect to $\left(x_1, x_2, y\right)$. Hence $x_1 g\left(x_1, x_2, y\right)$ can be regarded as a function of $(s, t)$, which will be denoted by $h(s, t)$. We easily find the quartic equation for $h=h(s, t)$:
\begin{align}\label{ueq}
\begin{split}
(256 s^3-128 s^2& t^2+16 s t (t^3+9)-4t^3-27) h^4+\left(4 t^3-16 s t+18\right) h^2-8h+1=0.
\end{split}
\end{align}
We specify the branches $h_j$ ($j = 1,2,3,4$) of the algebraic function $h$ near the
origin by their local behaviors:
\begin{align}\label{hk}
\begin{split}
h_1(s,t)&=\frac{1}{3}+\frac{4}{9}e^{-\frac{2\pi i}{3}}s+\frac{2}{9}e^{\frac{2\pi i}{3}}t+\cdots,\\
h_2(s,t)&=\frac{1}{3}+\frac{4}{9} e^{\frac{2\pi i}{3}}s+\frac{2}{9}e^{-\frac{2\pi i}{3}}t+\cdots,\\
h_3(s,t)&=\frac{1}{3}+\frac{4}{9}s+\frac{2}{9}t+\cdots,\\
h_4(s,t)&=-1-2st-4s^3+\cdots.
\end{split}
\end{align}
Note that each of the right-hand sides in \eqref{hk} is a convergent power series.
Now we specify the branches $g_j$ of $g$ by setting
\begin{align*}
g_j=h_j/x_1\qquad (j=1,2,3,4).
\end{align*}
Looking at the expansion of $x_1g_j$ at $(s,t)=(0,0)$,  we find that any three of $g_j$'s are linearly independent. Thus we have

\begin{thm}\label{thma}
The Borel transform $\psi_{\ell, B}$ of the WKB solution $\psi_{\ell}$ $(\ell=1,2,3)$ is written as a linear combination of any three of $g_k$'s. In particular, $\psi_{\ell, B}$'s are algebraic and hence $\psi_{\ell}$'s are resurgent and Borel summable in each Stokes region.
\end{thm}
Our next task is to write down $\psi_{\ell,B}$ explicitly in terms of $g_k$'s. Let us consider the restriction of $h$ to $t=0$. It satisfies
\begin{align}\label{eqphi}
\Phi_s(h):=\left(256 s^3-27\right) h^4+18 h^2-8 h+1=0.
\end{align}
Here we also use $h$ for $h(s, 0)$. Hence $h(s, 0)$ has a singularity at $s=p_{\ell}$ $(\ell=1,2,3)$. Taking the local expansions of the roots at $s=p_{\ell}$, we can specify the branches $h_j^{(\ell)}(s, 0)$ $(j=1,2,3,4)$ near $s=p_{\ell}$ as
\begin{align}
\begin{split}\label{hkl}
h_1^{(\ell)}(s,0)&=\frac{1}{2^{5/6}\sqrt{3}} e^{-\frac{2\pi i}{3}\ell} (p_\ell-s)^{-1/2}-\frac{2}{9}+\frac{7}{18 \cdot 2^{1/6}\sqrt{3}}e^{\frac{2\pi i}{3}\ell} (p_\ell-s)^{1/2}\\
&\hspace{20mm}-\frac{28 \cdot 2^{2/3}}{243}e^{-\frac{2\pi i}{3}\ell}(p_\ell-s)+\frac{385}{972\sqrt{6}}(p_\ell-s)^{3/2}+O((p_\ell-s)^2),\\
h_2^{(\ell)}(s,0)&=-\frac{1}{2^{5/6}\sqrt{3}} e^{-\frac{2\pi i}{3}\ell} (p_\ell-s)^{-1/2}-\frac{2}{9}-\frac{7}{18 \cdot 2^{1/6}\sqrt{3}} e^{\frac{2\pi i}{3}\ell} (p_\ell-s)^{1/2}\\
&\hspace{20mm}-\frac{28 \cdot 2^{2/3}}{243}e^{-\frac{2\pi i}{3}\ell}(p_\ell-s)-\frac{385}{972\sqrt{6}}(p_\ell-s)^{3/2}+O((p_\ell-s)^2),\\
h_3^{(\ell)}(s,0)&=\frac{4+i\sqrt{2}}{18}+O(p_\ell-s),\\
h_4^{(\ell)}(s,0)&=\frac{4-i\sqrt{2}}{18}+O(p_\ell-s).
\end{split}
\end{align}
Here the branches of the square roots are chosen as \eqref{p-s}.
For fixed $s\neq p_\ell$ sufficiently near $p_\ell$, $h(s,t)$ is holomorphic in $t$.
Hence the branches $h_j^{(\ell)}(s, 0)$ given above also specifies the branches $h_j^{(\ell)}(s, t)$ of $h(s, t)$ near $s=p_{\ell}$ if $|t|$ is sufficiently small.
The following lemma shows how these branches are related.
\begin{lmm}\label{branchh}
The branches $h^{(\ell)}_j$ $(j = 1,2,3,4;\ell = 1,2,3)$ and $h_j$ $(j = 1,2,3,4)$ satisfy the relations
\begin{align}\label{ghol}
\begin{split}
h_1(s,0)&= h_4^{(3)}(s,0)= h_2^{(1)}(s,0)= h_3^{(2)}(s,0),\\
h_2(s,0)&=h_3^{(3)}(s,0)= h_4^{(1)}(s,0)= h_2^{(2)}(s,0),\\
h_3(s,0)&=h_1^{(3)}(s,0)= h_3^{(1)}(s,0)= h_4^{(2)}(s,0),\\
h_4(s,0)&=h_2^{(3)}(s,0)= h_1^{(1)}(s,0)= h_1^{(2)}(s,0)
\end{split}
\end{align}
for $s \in \Sigma_1 \cup \Sigma_2 \cup \Sigma_3$. Here we set
\begin{align*}
\Sigma_1&= \{ e^{2\pi i/3} \sigma \mid  0<\sigma<p_3\},\\
\Sigma_2&= \{ e^{4\pi i/3}\sigma \mid  0<\sigma<p_3\},\\
\Sigma_3&= \{ \sigma \mid 0<\sigma<p_3 \}.
\end{align*}
           
\end{lmm}
\begin{proof}
We consider the roots $h$ of the quartic equation $\Phi_s(h)=0$ for $0<s<p_3$.
Eliminating $s$ from the relations
\begin{align*}
\frac{d}{dh}\Phi_s(h)=0, \quad \varphi=\Phi_s(h),
\end{align*}
we obtain $\varphi=(3h-1)^2$.
Therefore the critical values of $\Phi_s(h)$ are positive  for real critical points $h\neq 1/3$ and $\Phi_s(0)=1>0$. Note that  if $h=1/3$, then $s=0$, namely $s \not \in \Sigma_3$.
There are at most 3 real critical points of $\Phi_s(h)$ and each of such critical values is positive.
Since $\Phi_s(h)$ is a polynomial in $h$ with real coefficients, $\Phi_s(h) \to -\infty$ as $h \to \pm\infty$, and $\Phi_s(0) = 1$, there exist $\lambda_4 < 0$ and $\lambda_3 > 0$ such that $\Phi_s(\lambda_3) = \Phi_s(\lambda_4) = 0$ and $\Phi_s(h) \neq 0$ for $h < \lambda_4$ or $h > \lambda_3$. Furthermore, from the properties of critical values, $\Phi_s(h) \neq 0$ holds for $\lambda_4 < h < \lambda_3$. Therefore, since $\Phi_s(h)$ is a quartic polynomial in $h$ with real coefficients, its roots are $\lambda_1 \pm i\lambda_2, \lambda_3, \lambda_4$ where $\lambda_1, \lambda_2 \in \mathbb{R}$ and $\lambda_2\neq 0$. Hence $\Phi_s(h)$ can be written as
\begin{align}\label{eqphi2}
\Phi_s(h)=\left(256 s^3-27\right)(h-(\lambda_{1}+i\lambda_{2}))(h-(\lambda_{1}-i\lambda_{2}))(h-\lambda_{3})(h-\lambda_{4}) \quad (0<s<p_3).
\end{align}
Comparing the third-order and first-order coefficients of \eqref{eqphi} and \eqref{eqphi2} with respect to $h$, respectively, we obtain
\begin{align}\label{h3coeff}
0=&-( 256 s^3-27)(\lambda_3 + \lambda_4 + 2 \lambda_1),\\\label{h2coeff}
0>-8=&- (256 s^3-27)((\lambda_3 + \lambda_4) ( \lambda_2^2 +  \lambda_1^2) + 2 \lambda_1 \lambda_3  \lambda_4 ).
\end{align}
Since \(256s^3 - 27 < 0\) for $0<s<p_3$, equation \eqref{h3coeff} yields
\begin{align}\label{h3coeff1}
\lambda_3+\lambda_4=-2\lambda_1.
\end{align}
Moreover, inequality \eqref{h2coeff} implies
 \begin{align}\label{h2coeff1}
 (\lambda_3 + \lambda_4) ( \lambda_2^2 +  \lambda_1^2) + 2 \lambda_1 \lambda_3  \lambda_4 <0.\end{align}
Substituting \eqref{h3coeff1} into \eqref{h2coeff1}, we obtain
\begin{align}\label{lambdaineq}
(\lambda_3 + \lambda_4) ( \lambda_2^2 +  \lambda_1^2) + 2  \lambda_3  \lambda_4 \lambda_1=-2\lambda_1( \lambda_2^2 +  \lambda_1^2- \lambda_3  \lambda_4)<0.
\end{align}
By \eqref{lambdaineq}, we obtain \(\lambda_1 > 0\) for \(0 < s < p_3\). Note that since \(\lambda_3 > 0\) and \(\lambda_4 < 0\), we have \(-\lambda_3 \lambda_4 > 0\).  
Therefore, the quartic equation $\Phi_s(h)=0$ of $h$ has one positive root, one negative root and two imaginary roots with positive real part for $0<s<p_3$.
When $s\rightarrow +0$,
\begin{align}\label{hbis1}
h_3(s,0)>0, \quad h_4(s,0)\rightarrow-1<0
\end{align}
and if $0<s\ll1$, we have
\begin{align}\label{hbis2}
\operatorname{Re}h_1(s,0)>0,\quad \operatorname{Im}h_1(s,0)<0,\quad
\operatorname{Re}h_2(s,0)>0,\quad\operatorname{Im}h_2(s,0)>0.
\end{align}
On the other hand, when $s\rightarrow p_3-0$, we have
\begin{align}\label{hbis3}
h_1^{(3)}(s,0) \rightarrow+\infty,& \quad h_2^{(3)}(s,0) \rightarrow-\infty,\\ \label{hbis4}
\operatorname{Re} h_3^{(3)}(s,0)>0,\quad \operatorname{Im} h_3^{(3)}(s,0)>0,&\quad
\operatorname{Re} h_4^{(3)}(s,0)>0,\quad \operatorname{Im} h_4^{(3)}(s,0)<0.
\end{align}
Comparing \eqref{hbis1} with \eqref{hbis3} and \eqref{hbis2} with \eqref{hbis4}, respectively, we get
\begin{align*}
h_1(s,0)=h_4^{(3)}(s,0), \quad h_2(s,0)=h_3^{(3)}(s,0), \quad h_3(s,0)=h_1^{(3)}(s,0),\quad  h_4(s,0)=h_2^{(3)}(s,0)
\end{align*}
for $0<s<p_3$. Since we have 
\[
\Phi_{e^{\frac{2\pi i}{3}\ell}\sigma}(h)=\Phi_\sigma(h)\qquad (\ell=1,2)\]
for $\sigma \in \Sigma_3$, we obtain the rest of \eqref{ghol}.
\end{proof}
We set $g_j^{(\ell)}=h_j^{(\ell)}/x_1$ ($j=1,2,3,4; \ell=1,2,3$). 
By using Lemma \ref{branchh}, we obtain 
\begin{align}\label{ghollabel}
\begin{split}
g_1&= g_4^{(3)}= g_2^{(1)}= g_3^{(2)},\\
g_2&=g_3^{(3)}= g_4^{(1)}= g_2^{(2)},\\
g_3&=g_1^{(3)}= g_3^{(1)}= g_4^{(2)},\\
g_4&=g_2^{(3)}= g_1^{(1)}= g_1^{(2)}
\end{split}
\end{align}
for sufficiently small $|t|$.\par
Recall that the branches of $x_{1}g_{j}$, $x_{1}g_{j}^{(\ell)}$ and $x_{1}\psi_{\ell,B}$ are specified locally by \eqref{hk}, \eqref{hkl} and \eqref{WKBBorel}, respectively. 
Let $f$ be one of $g_{j}, g_j^{(\ell)}, \psi_{\ell, B}$ ($j=1,2,3,4; \ell=1,2,3$).
For the analytic continuation of $f$ with respect to $x$,
we continue the roots $\zeta_\ell$ and the points $u_\ell$ ($\ell=1,2,3$)
along the same path, using the labeling fixed in \S3.
Suppose that $|y/x_{1}^{4/3}|$ and $|x_{2}/x_{1}^{2/3}|$ are sufficiently small.
For a sufficiently small $\delta>0$ and $k=1,2,3$, we set
\[
\Omega_{k}=\{y\,|\, |y-u_{k}|<\delta\,\}\cup\{y\,|\, {\rm Re}\, y>{\rm Re}\, u_{k}, |{\rm Im} (y-u_{k})|<\delta\,\}
\]
and
\[
L_{k}=\{u_{k}+r\,|\, r>0\}.
\]
For the case $f=g_{j}$, we take the analytic continuation  $\tilde{f}$ of $f$ (in the $y$ variable) along the segment connecting the origin and $u_{k}$, which is well-defined by our assumptions.
More precisely, for sufficiently small $|t|$, there does not exist a point $u_j$ with $j\neq k$ on the segment $0u_k$.
Indeed, at $t=0$,  the points $u_1,u_2,u_3$ are the vertices of an equilateral triangle centered at $0$ and they depend continuously on $t$.
For the case $f=g_{j}^{(\ell)}$ or $f=\psi_{\ell, B}$, we take the analytic continuation $\tilde{f}$ of $f$ along the segment $u_{\ell}u_{k}$ connecting $u_{\ell}$ and $u_{k}$ $(k\neq \ell$).
In both cases, $\tilde{f}$ defines a single-valued holomorphic function in $\Omega_{k}\setminus L_{k}$ since $f$ is an algebraic function with the branch points $\{u_{1}, u_{2}, u_{3}\}$.
Then, for $y\in L_{k}$, we set
\[
\Delta_{u_{k}}f(y)=\lim_{\varepsilon\rightarrow+0}(\tilde{f}(y+i\varepsilon)-\tilde{f}(y-i\varepsilon)).
\]
For the case $f=g_{j}^{(\ell)}$ or $f=\psi_{\ell, B}$, we have to consider the analytic continuation with respect to the $x$ variable since $u_{\ell}$ depends on $x$. In this case, we regard $f$ as an analytic function of $x$ containing $y$ as an analytic parameter and  we always take the analytic continuation of the germ of $f$ specified by \eqref{WKBBorel}, \eqref{hk} or \eqref{hkl}.
If $u_{1}, u_{2}, u_{3}$ are distinct and not colinear, we can take the analytic continuation $\tilde{f}$ 
of thus obtained function in the $y$ variable along $u_{\ell}u_{k}$ and define $\Delta_{u_{k}}f(y)$ as above.
In all cases, we regard the analytic continuation of $\Delta_{u_{k}}f(y)$ to ${\mathbb C}_{y}\setminus \{u_{1}, u_{2}, u_{3}\}$ as a multi-valued function of $y$, denote it by $\Delta_{u_{k}}f(y)$ and call it the discontinuity of $f$ along $L_{k}$.
\begin{thm}\label{thmpe}
Suppose that $|x_2/x_1^{2/3}|$ is sufficiently small.

\begin{itemize}
\item[{\rm (i)}]
The Borel transform $\psi_{\ell, B}$ of the WKB solution is expressed in terms of $g_{k}$'s as
\begin{equation}\label{thmpe1}
\psi_{\ell,B}=\frac{i}{\sqrt{\pi}}(-1)^{\ell-1}(g_{\ell}-g_{4})\quad(\ell=1,2,3).
\end{equation}
This relation \eqref{thmpe1} holds when $ |y/x_1^{4/3}| $ is sufficiently small. 
\item[{\rm (ii)}]
The discontinuity of $g_{4}$ along $L_{u_{\ell}}$ reproduces $\psi_{\ell,B}$, that is,
\begin{align}\label{thm5.1dc}
\psi_{\ell,B}=\frac{i}{\sqrt{\pi}}\Delta_{u_{\ell}}g_{4}
\end{align}
holds for $\ell =1,2,3$. 
\end{itemize}
\end{thm}
Note that, since both sides of \eqref{thmpe1} are algebraic functions, taking the analytic continuation into account, \eqref{thmpe1} also holds outside the singular locus of $M_B$. 
\begin{proof}
(i) By using Theorem \ref{thma}, we obtain
\begin{align}\label{temp4}
x_1\psi_{\ell,B}=\sum_{i=1}^4 C_i^{(\ell)} h_i^{(\ell)}(s,t),
\end{align}
where $C_k^{(\ell)}$ $(k=1,2,3,4)$ is a constant independent of $s$ and $t$.
Comparing the coefficients of the right-hand side of \eqref{temp4} and  the right-hand side of \eqref{WKBBorel} with respect to $(s-p_\ell)$, and using \eqref{hkl}, we have 
\begin{align}\label{psig}
x_1 \psi_{1,B}\Big|_{t=0}&=-\frac{i}{\sqrt{\pi}}(h_1^{(1)}(s,0)-h_2^{(1)}(s,0)),\\
x_1 \psi_{2,B}\Big|_{t=0}&=\frac{i}{\sqrt{\pi}}(h_1^{(2)}(s,0)-h_2^{(2)}(s,0)),\\
x_1 \psi_{3,B}\Big|_{t=0}&=\frac{i}{\sqrt{\pi}}(h_1^{(3)}(s,0)-h_2^{(3)}(s,0)).
\end{align}
By using Lemma \ref{branchh}, we obtain 
\begin{align}\label{psig2}
x_1 \psi_{1,B}\Big|_{t=0}&=\frac{i}{\sqrt{\pi}}(h_1(s,0)-h_4(s,0)),\\
x_1 \psi_{2,B}\Big|_{t=0}&=-\frac{i}{\sqrt{\pi}}(h_2(s,0)-h_4(s,0)),\\
x_1 \psi_{3,B}\Big|_{t=0}&=\frac{i}{\sqrt{\pi}}(h_3(s,0)-h_4(s,0)).
\end{align}
Since $x_1 \psi_{\ell,B}$ $(\ell=1,2,3)$ are holomorphic at $(s,t)=(0,0)$, we can obtain \eqref{thmpe1}.\\
(ii)  
We consider the case $\ell=1$
and suppose that $x_{1}>0$. It suffices to compute
the discontinuity of $h_{4}(s,0)$ in the $s$-variable along $L_{p_{1}}=\{p_1+r \mid r>0\}$.
It follows from \eqref{ghol} and \eqref{hkl} that 
$h_{4}(s,0)$ has an expansion of the form
\begin{align*}
\frac{1}{2^{5/6}\sqrt{3}}e^{-\frac{2\pi i}{3}}(p_{1}-s)^{-1/2}-\frac{2}{9}+\frac{7}{18 \cdot 2^{1/6}\sqrt{3}}e^{\frac{2\pi i}{3}} (p_1-s)^{1/2}\\
-\frac{28 \cdot 2^{2/3}}{243}e^{-\frac{2\pi i}{3}}(p_1-s)+\frac{385}{972\sqrt{6}}(p_1-s)^{3/2}
+\cdots
 \end{align*}
for $s\in \Sigma_{1}$ near $s=p_{1}$. We have to take 
the analytic continuation in $s$ of $h_{4}(s,0)$ to $p_{1}+\delta$ in two ways, where $\delta>0$ is sufficiently small.
We set
\[
p_{1}-s=\delta e^{(2\pi /3+\phi)i} \quad \left(-\frac{5}{3}\pi<\phi\le \frac{\pi}{3}\right)
\]
in the above expansion. Since
\[
 \lim_{\phi \to -5\pi/3+0} h_{4}(p_{1}-\delta e^{(2\pi /3+\phi)i},0)=\lim_{\varepsilon \to +0} h_{4}(p_{1}+\delta+ i\varepsilon,0),
\]
we have
\begin{align*}
 \lim_{\varepsilon \to +0} h_{4}(p_{1}+\delta+ i\varepsilon,0)&=\frac{1}{2^{5/6}\sqrt{3}}e^{-\frac{2\pi i}{3}}\delta^{-1/2}e^{\pi i/2}-\frac{2}{9}+\frac{7}{18 \cdot 2^{1/6}\sqrt{3}}e^{\frac{2\pi i}{3}}\delta^{1/2}e^{-\pi i/2}\\
&\hspace{25mm} -\frac{28 \cdot 2^{2/3}}{243}e^{-\frac{2\pi i}{3}}\delta e^{-\pi i} +\frac{385}{972\sqrt{6}}\delta^{3/2}e^{-3\pi i/2}+\cdots.
 \end{align*}
Similarly, since
\[
 \lim_{\phi \to \pi/3-0} h_{4}(p_{1}-\delta e^{(2\pi /3+\phi)i},0)=\lim_{\varepsilon \to +0} h_{4}(p_{1}+\delta- i\varepsilon,0),
\]
we have
 \begin{align*}
 \lim_{\varepsilon \to +0}h_{4}(p_{1}+\delta- i\varepsilon,0)&=\frac{1}{2^{5/6}\sqrt{3}}e^{-\frac{2\pi i}{3}}\delta^{-1/2}e^{-\pi i/2}-\frac{2}{9}+\frac{7}{18 \cdot 2^{1/6}\sqrt{3}}e^{\frac{2\pi i}{3}} \delta^{1/2}e^{\pi i/2}\\
&\hspace{25mm} -\frac{28 \cdot 2^{2/3}}{243}e^{-\frac{2\pi i}{3}}\delta e^{\pi i} +\frac{385}{972\sqrt{6}}\delta^{3/2}e^{3\pi i/2}+\cdots.
\end{align*}
Thus we have
\begin{align}\label{thm5.5_h}
\begin{split}
\Delta_{p_{1}}h_{4}(p_{1}+\delta,0)&= \lim_{\varepsilon \to +0}\left(h_{4}(p_{1}+\delta+ i\varepsilon,0)-h_{4}(p_{1}+\delta- i\varepsilon,0)\right)\\
&=i\frac{2^{1/6}}{\sqrt{3}}e^{-\frac{2\pi i}{3}}\delta^{-1/2}-i\frac{7}{9 \cdot 2^{1/6}\sqrt{3}}e^{\frac{2\pi i}{3}}\delta^{1/2}+i \frac{385}{486\sqrt{6}}\delta^{3/2}+\cdots.
\end{split}
\end{align}
On the other hand, using \eqref{WKBBorel}, we obtain
\begin{align}\label{thm5.5_x1psi}
x_{1}\psi_{1,B}\big|_{(s,t)=(p_{1}+\delta,0)}=-\frac{2^{1/6}}{\sqrt{3\pi}}e^{-\frac{2\pi i}{3}}\delta^{-1/2}+\frac{7}{9 \cdot 2^{1/6}\sqrt{3\pi}}e^{\frac{2\pi i}{3}}\delta^{1/2}-\frac{385}{486\sqrt{6\pi}}\delta^{3/2}+\cdots.
\end{align}
Comparing the coefficients with respect to $\delta$ in both sides of \eqref{thm5.5_h} and \eqref{thm5.5_x1psi}, we have
\begin{align*}
\frac{i}{\sqrt{\pi}}\Delta_{p_{1}}h_{4}=x_{1}\psi_{1,B}\big|_{t=0}.
\end{align*}

For a fixed $(x_{1}, y)$ satisfying $0<|y/x_{1}^{4/3}-p_{1}|$ being sufficiently small, 
$x_{1}\psi_{1,B}$ is a holomorphic function of $t=x_{2}/x_{1}^{2/3}$ for $|t|$ being sufficiently small
and $\psi_{1,B}$ can be written as a linear combination of branches of analytic continuation of $g_{4}$.
Therefore we obtain
\[
\frac{i}{\sqrt{\pi}}\Delta_{u_{1}}g_{4}=\psi_{1,B}.
\]
This holds in the domain of holomorphy of both members if the analytic continuation is taken
consistently. Other cases are shown similarly.
\end{proof}

\begin{rem}
Restricting the Pearcey system to $x_{2}=c \ (\neq 0)$, where $c$ is a constant, we can treat the equation given by Berk-Nevins-Roberts \cite{BNR}
in a similar manner to our discussion given above. 
By using $P_{3}u=(\eta\partial_{2}-\partial_{1}^{2})u=0$, we may replace $\partial_{2}u$ by $\eta^{-1}\partial_{1}^{2}u$ in $P_{1}u= P_{2}u=P_{4}u=0$. Setting $x_{2}=c$ and $x_{1}=x$ after this replacement and multiplying some powers of $\eta$, we have the system of differential equations
\begin{align*}
P_{1}'u&=(4\partial_{x}^{3}+2\eta^{2}c\partial_{x}+\eta^{3}x)u=0,\\
P_{4}'u&=(2c\partial_{x}^{2}+3x\eta \partial_{x}-4\eta^{2}\partial_{\eta}-\eta)u=0.
\end{align*}
Note that $P_{2}u=0$ yields $\partial_{x}P_{1}'u=0$.
Hence this system characterizes the restriction $u(x,c,\eta)$ of \eqref{intPe}. Since these operators do not form a Gr\"obner basis
of the left ideal generated by the operators, 
we use the algorithm of computing the Gr\"obner basis (\cite{Oaku2, Oaku}) from $(P_{1}', P_{4}')$ and have
 \begin{align*}
\begin{split}
R_1&=32 c  \eta^2 \partial_\eta
   ^2+(8 c^3 \eta x -6 c
    x  +27 \eta  x^3) \partial_x\\
   &\qquad\qquad 
 +(8 c^3  \eta^2+32 c \eta -36 \eta ^2 x^2 )\partial_\eta+4 c^3 \eta +6
   c^2 \eta ^2 x^2-9 \eta 
   x^2-2 c,\\
\end{split}\\
\begin{split}
R_2&=8 c  \eta \partial_\eta  \partial_x+(2 c^3 \eta  -4 c+9\eta  x^2 )\partial_x -12 \eta ^2 x\partial_\eta+c^2 \eta ^2
   x -3 \eta  x,\\
\end{split}\\
\begin{split}
R_3&=2 c \partial_{x}^{2}+3 x \eta\partial_{x}-4 \eta^2 \partial_{\eta}-\eta.\\
\end{split}
\end{align*}
We consider the system of differential equations
\begin{equation}\label{7.2}
R_{i}\psi=0\quad (i=1,2,3),
\end{equation}
or more precisely, we consider the $D$-module in the variables ($x,\eta$) defined by the left ideal generated by $R_{i}\ (i=1,2,3)$
and call this the BNR system, which characterizes $u(x,c,\eta)$.
Restricting our discussions concerning the Pearcey system to $x_{2}=c$, we obtain the counterparts for the BNR system. 
It can be seen from the discussion in \cite[Theorem A.1.1]{HKT}, \cite{TakeiRIMS}  that the WKB solutions to the BNR equation $(4\partial_x^3+2c \eta^2 \partial_x+x\eta^3)\psi=0$ $(c\neq0)$ are Borel summable. 
Moreover, using the discussion, we can see that suitably normalized WKB solutions are resurgent.
Applying our method to the BNR system, we can specify the normalization of the WKB solutions which are resurgent. 
\end{rem}

\section{Analytic continuation of the Borel transforms and connection formulas}
Lemma \ref{branchh} and Theorem \ref{thmpe} enable us to take analytic continuation of $\psi_{\ell,B}$ to the possible
singularities $u_{k}$ $(k\neq \ell)$. We suppose that $|t|$ is sufficiently small.  
For an analytic function germ (possibly multi-valued) $f$ at $y=u_{\ell}$, let $c^{\ast}_{\ell k}f$
denote the analytic continuation of $f$ along the oriented segment $u_{\ell}u_{k}$ 
from $u_\ell$ to $u_k$
in $y$ variable.
If $f$ has a square root type singularity at $y=u_{\ell}$, we denote by $c_{\ell}^{\ast}f$ 
another branch of $f$. If $f$ is holomorphic at $y=u_{\ell}$, we set 
$c_{\ell}^{\ast}f=f$. Using \eqref{ghollabel}, we can take analytic continuation of $g^{(\ell)}_{j}$ to the possible singularity $u_{k}$:
\begin{equation}\label{cjkg}
\left\{
\begin{split}
c_{12}^{\ast}g_{1}^{(1)}&=g^{(2)}_{1}, &c_{13}^{\ast}g^{(1)}_{1}&=g^{(3)}_{2},& c_{23}^{\ast}g^{(2)}_{1}&=g^{(3)}_{2},\\
c_{12}^{\ast}g_{2}^{(1)}&=g^{(2)}_{3}, & c_{13}^{\ast}g^{(1)}_{2}&=g^{(3)}_{4},& c_{23}^{\ast}g^{(2)}_{2}&=g^{(3)}_{3},\\
c_{12}^{\ast}g_{3}^{(1)}&=g^{(2)}_{4}, & c_{13}^{\ast}g^{(1)}_{3}&=g^{(3)}_{1},& c_{23}^{\ast}g^{(2)}_{3}&=g^{(3)}_{4},\\
c_{12}^{\ast}g_{4}^{(1)}&=g^{(2)}_{2}, & c_{13}^{\ast}g^{(1)}_{4}&=g^{(3)}_{3},& c_{23}^{\ast}g^{(2)}_{4}&=g^{(3)}_{1},
\end{split}
\right.
\end{equation}
\begin{equation}\label{cjkg2}
\left\{
\begin{split}
c_{1}^{\ast}g^{(1)}_{1}&=g^{(1)}_{2},&c_{1}^{\ast}g^{(1)}_{2}&=g^{(1)}_{1},&c_{1}^{\ast}g^{(1)}_{3}&=g^{(1)}_{3},&
c_{1}^{\ast}g^{(1)}_{4}&=g^{(1)}_{4},\\
c_{2}^{\ast}g^{(2)}_{1}&=g^{(2)}_{2},&c_{2}^{\ast}g^{(2)}_{2}&=g^{(2)}_{1},&c_{2}^{\ast}g^{(2)}_{3}&=g^{(2)}_{3},&
c_{2}^{\ast}g^{(2)}_{4}&=g^{(2)}_{4},\\
c_{3}^{\ast}g^{(3)}_{1}&=g^{(3)}_{2},&c_{3}^{\ast}g^{(3)}_{2}&=g^{(3)}_{1},&c_{3}^{\ast}g^{(3)}_{3}&=g^{(3)}_{3},&
c_{3}^{\ast}g^{(3)}_{4}&=g^{(3)}_{4}.\\
\end{split}
\right.
\end{equation}
We consider the case where $x_{1}^{4/3}>0$. 
The left-hand side of \eqref{ueqq} vanishes at $(x_1, x_2, u)=(x_1, 0, p_\ell x_1^{4/3})$. Hence $u_\ell$ is sufficiently close to $p_\ell x_1^{4/3}$ ($\ell=1,2,3$) if $x_2$ is sufficiently small. 
In order to consider the Borel sum of $\psi_{\ell}$,
we take the half lines in the 
$y$-plane starting at $u_{\ell}$ 
with the positive real direction as new branch cuts of $\psi_{\ell,B}$. Consequently, we choose the branch of $\psi_{\ell,B}$ near $y=u_{\ell}$ as $0\le \arg(y-u_{\ell})<2\pi$. For $\psi_{\ell,B}$, we abbreviate
$\Delta_{u_{k}}c^{\ast}_{\ell k}\psi_{\ell,B}$ ($k\neq \ell$) as $\Delta_{u_{k}}\psi_{\ell,B}$ and
$\Delta_{u_{j}}c^{\ast}_{kj}c^{\ast}_{k}c^{\ast}_{\ell k}\psi_{\ell,B}$ as $\tilde{\Delta}_{u_{j}}\psi_{\ell,B}$
for distinct $\ell, j, k$.

\begin{thm}\label{thmdis}
Fix a point $q=\left(q_1, 0\right)\left(q_1>0\right)$. 
For a point $\overset{\circ}{x}\in{\mathbb C}^{2}\setminus{T}$ and distinct $ j, k, \ell$, we assume 
that the point $u_{\ell}\left(x\right)$ does not cross the segment $u_j\left(x\right) u_k\left(x\right)$ 
with endpoints $u_{j}\left(x\right) $ and $u_k\left(x\right)$ when a point $x$ moves along 
the segment $q \overset{\circ}{x}$.
Then we have
\begin{align}\label{6.1eq} 
\Delta_{u_{m}}\psi_{n,B}&=(-1)^{n}\psi_{m,B},\\ \label{6.2eq} 
\tilde{\Delta}_{u_{m}}\psi_{n,B}&=0
\end{align}
for $ m, n=1,2,3;n\neq m, x\in q \overset{\circ}{x}$.
\end{thm}
\begin{proof}
We prove the case where $n=1, m=3$. 
Other cases can be proved similarly. 
By our assumption, \eqref{ghollabel} 
and Theorem \ref{thmpe}
hold for all $x\in q\overset{\circ}{x}$ and hence \eqref{cjkg} and \eqref{cjkg2} also hold
since the functional relations are kept under analytic the continuation in $x$-variable. 
Combining 
\eqref{ghollabel}
and Theorem 
\ref{thmpe},
we have
\[
\psi_{1,B}=\frac{i}{\sqrt{\pi}}(g_{1}-g_{4})=\frac{i}{\sqrt{\pi}}(g_{2}^{(1)}-g_{1}^{(1)}).
\]
Using \eqref{cjkg} and \eqref{cjkg2},
we have
\begin{align*}
c^{\ast}_{13}\psi_{1,B}&=\frac{i}{\sqrt{\pi}}(g_{4}^{(3)}-g_{2}^{(3)}),\\
c^{\ast}_{23}c^{\ast}_{2}c^{\ast}_{12}\psi_{1,B}&=c^{\ast}_{23}c^{\ast}_{2}\frac{i}{\sqrt{\pi}}(g^{(2)}_{3}-g^{(2)}_{1})=c^{\ast}_{23}\frac{i}{\sqrt{\pi}}(g^{(2)}_{3}-g^{(2)}_{2})=\frac{i}{\sqrt{\pi}}(g^{(3)}_{4}-g^{(3)}_{3}).
\end{align*}
On the other hand, it follows from
\eqref{hkl} and \eqref{ghollabel}
that $\Delta_{u_{3}}g^{(3)}_{3}=\Delta_{u_{3}}g^{(3)}_{4}=0$ and $i/\sqrt{\pi}\Delta_{u_{3}}g^{(3)}_{2}=
i/\sqrt{\pi}\Delta_{u_{3}}g_{4}=\psi_{3,B}$. Then we have
\begin{align*}
\Delta_{u_{3}}\psi_{1,B}&=\frac{i}{\sqrt{\pi}}\Delta_{u_{3}}(g_{4}^{(3)}-g_{2}^{(3)})=-\psi_{3,B},\\
\tilde{\Delta}_{u_{3}}\psi_{1,B}&=\frac{i}{\sqrt{\pi}}\Delta_{u_{3}}(g^{(3)}_{4}-g^{(3)}_{3})=0.
\end{align*}
Hence we obtain \eqref{6.1eq} and \eqref{6.2eq} for $n=1$, $m=3$.
\end{proof}
To track the singularities of the Borel transform of WKB solutions to a (system of) higher-order linear 
(partial) differential equation
with polynomial coefficients, it is necessary 
to evaluate integrals of algebraic functions in general.
However, in our case, since the singularities correspond to the roots $u_k$ 
of the algebraic equation \eqref{ueqq}, we can track them as $x_1$, $x_2$ vary without 
using non-elementary transcendental functions.\par

Next we consider the case where $\overset{\circ}{x}$ is far from $q$
and the path of analytic continuation in the $x$-variable is a piecewise smooth curve
connecting $q$ and $\overset{\circ}{x}$ in Theorem \ref{thmdis}.
If the paths of analytic continuation $c^{\ast}_{\ell k}$ and $c_{kj}^{\ast}c_{k}^{\ast}c_{\ell k}^{\ast}$ (in the $y$-variable)
of $\psi_{\ell,B}$'s are deformed suitably as $x$ moves, the discontinuity formulas \eqref{6.1eq} and \eqref{6.2eq} keep hold.
To state this observation more precisely, we define a modification of the notion of discontinuity.
We use the same notation as that given in the definition of $\Delta_{u_{k}}$ in \S5.
We assume that $x\in{\mathbb C}^{2}\setminus{T}$. For distinct $\ell, k, j; \{\ell, k, j\}=\{1,2,3\}$,
let $l$ be an oriented smooth curve starting from $u_{\ell}$ and ending at $u_{k}$ which does not pass $u_{j}$.
For $f=g_{m}^{(\ell)}$ ($m=1,2,3,4$) or $f=\psi_{\ell, B}$, we take the analytic continuation $\hat{f}$ of $f$ along $l$,
which defines a single-valued holomorphic function in $\Omega_{k}\setminus L_{k}$. For $y\in L_{k}$, we set
\[
\Delta_{u_{k},l}'f(y)=\lim_{\varepsilon\rightarrow+0}(\hat{f}(y+i\varepsilon)-\hat{f}(y-i\varepsilon)).
\]
This function is analytically continued to ${\mathbb C}_{y}\setminus \{u_{1},u_{2},u_{3}\}$ as a multi-valued
function of $y$, which is also denoted by $\Delta_{u_{k},l}'f(y)$. 
We call  $\Delta_{u_{k},l}'f(y)$ the discontinuity at $u_{k}$ of the analytic continuation of $f$ along $l$.
Analytic continuation in $x$-variable of the relations given in Theorem \ref{thmdis} yields\\

\begin{thm}\label{thmdisp}
Let $q=(q_{1},0) \ (q_{1}>0)$ and $\overset{\circ}{x}\in{\mathbb C}^{2}\setminus T$.
For $x=q$, let $l_{jk}$ be the oriented segment $u_{j}u_{k}$ $(j, k=1,2,3; j\neq k)$ and
$\hat{l}_{jk}$ a curve starting from $u_{j}$, ending at $u_{k}$ and detouring $u_{\ell}$ $(j, k, \ell=1,2,3; j, k, \ell\colon\!\mbox{distinct})$. More precisely, $\hat{l}_{jk}$ is corresponding to a path of the analytic continuation $c_{\ell k}^{\ast}c_{\ell}^{\ast}c_{j\ell}^{\ast}$.
Let $\gamma$ be
a piecewise smooth path in ${\mathbb C}^{2}\setminus T$ starting from $q$ and ending at $\overset{\circ}{x}$. When $x$ moves
along $\gamma$, we deform $l_{jk}$ and $\hat{l}_{jk}$ continuously so that $u_{\ell}$ do not cross 
them throughout the deformation $(j, k, \ell=1,2,3; j, k, \ell\colon\!\mbox{distinct})$. At $x\in\gamma$, we obtain deformed curves
$l_{jk}'$ and $\hat{l}_{jk}'$, respectively. Then we have
\begin{align}\label{6.1eqp}
\Delta_{u_{m},l}'\psi_{n,B}&=(-1)^{n}\psi_{m,B} \ \ \mbox{for}\ l=l_{nm}', \\ \label{6.2eqp}
{\Delta}_{u_{m},l}'\psi_{n,B}&=0 \ \ \mbox{for}\ l=\hat{l}_{nm}'. 
\end{align}
for $ m, n=1,2,3;n\neq m, x\in\gamma$.
\end{thm}
Using this theorem, we can deduce, in principle, connection formulas of WKB solutions $\psi_{\ell}$ 
across the Stokes set in a neighborhood of arbitrary generic point on the Stokes set.\par
The following figures show how $u_k$ ($k=1,2,3$) move as $x=(x_{1},x_{2})$ starts from $x^{(1)}=(0.15,0)$
and goes to $x^{(12)}=(0.28+0.69 i, (1+i)/2)$ along a polygonal line consisting of the segments $x^{(k)}x^{(k+1)}$
($k=1,2,\dots, 11$). 
This polygonal line crosses the following types of Stokes surface:
\[
2<1,\quad 2<3,\quad 3<1.
\]
If the analytic continuation of the Borel sum of the WKB solution with respect to such a polygonal line can be computed, the connection formulas for the general case can also be computed in a similar manner.
For $x=x^{(1)}$, let $l_{jk}$ denote the oriented segment $u_{j}u_{k}$ ($j, k=1,2,3; j\neq k$). 
Some of them are shown in Figure \ref{6.1}. For $x=x^{(m)}$ ($m=2,3,\dots,13$), we denote by $l_{jk}'=l_{jk}'(x^{(m)})$
the deformed curve of $l_{jk}$ which is defined in Theorem \ref{thmdisp} for the case where
$q=x^{(1)}$ and $\gamma$ to
be our polygonal line (see Figures \ref{6.2}-\ref{6.12}).  

\begin{figure}[H]
  \begin{minipage}[b]{0.32\linewidth}
    \centering
\includegraphics[width=43mm]{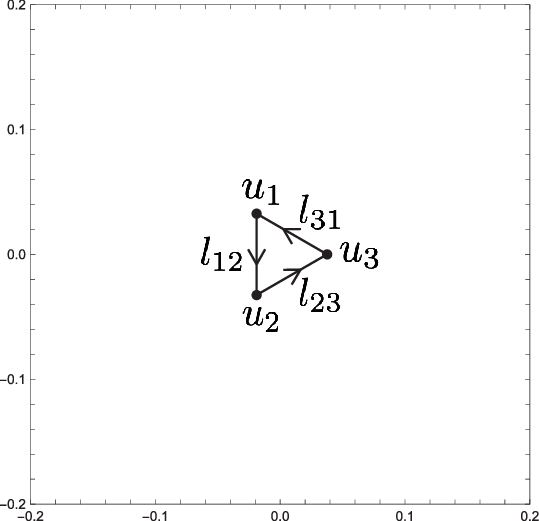}
    \caption{ $ x^{(1)}=(0.15,0)$}\label{6.1}
  \end{minipage}
  \begin{minipage}[b]{0.32\linewidth}
    \centering
\includegraphics[width=43mm]{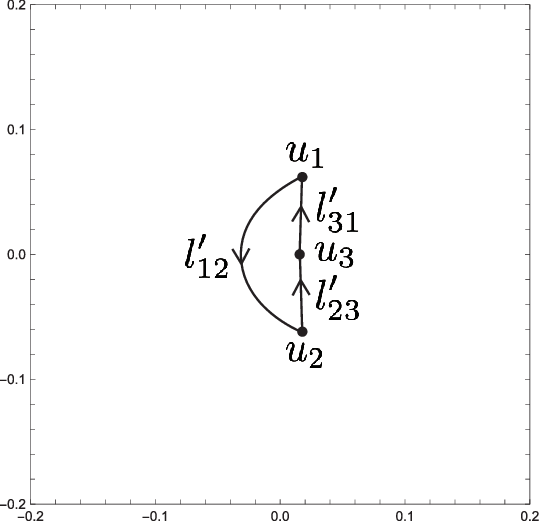}
    \caption{$x^{(2)}=(0.15,0.32)$}\label{6.2}
  \end{minipage}
    \begin{minipage}[b]{0.32\linewidth}
    \centering
\includegraphics[width=43mm]{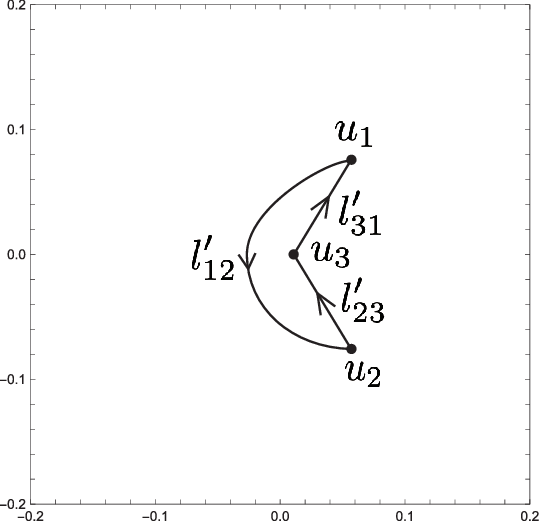}
    \caption{$x^{(3)}=(0.15,0.5)$}\label{6.3}
  \end{minipage}
\end{figure}
\begin{figure}[H]
  \begin{minipage}[h]{0.32\linewidth}
    \centering
\includegraphics[width=43mm]{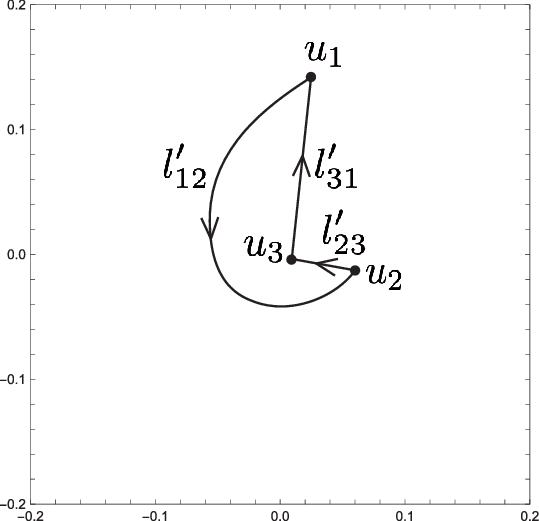}
    \caption{${x^{(4)}=(0.15,\frac{2+i}{4})}$}\label{6.4}
  \end{minipage}
  \begin{minipage}[h]{0.32\linewidth}
    \centering
\includegraphics[width=43mm]{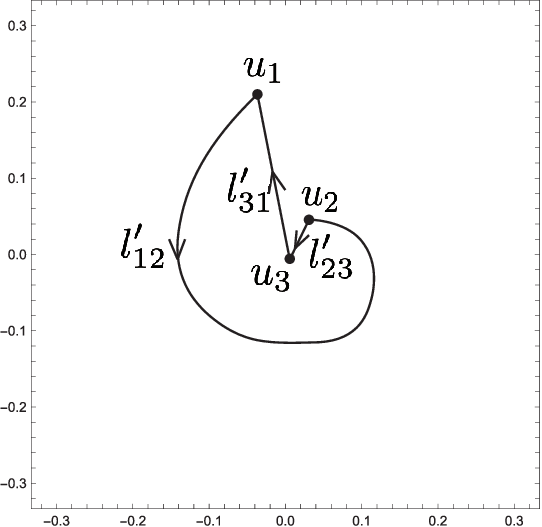}
    \caption{$x^{(5)}=(0.15,\frac{1+i}{2})$}\label{6.5}
  \end{minipage}
    \begin{minipage}[h]{0.32\linewidth}
    \centering
\includegraphics[width=43mm]{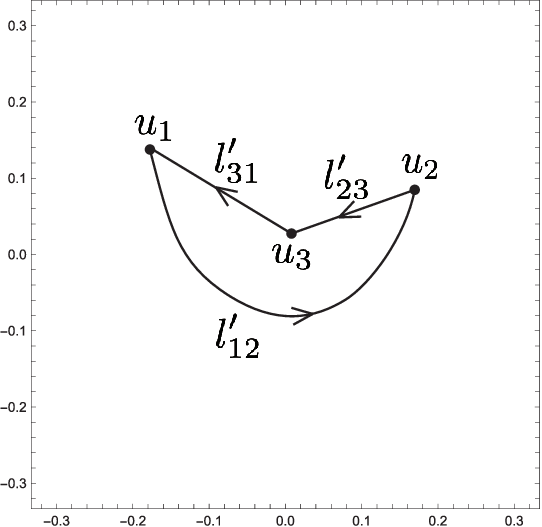}
    \caption{$x^{(6)}=(0.15+0.25i,\frac{1+i}{2})$}\label{6.6}
  \end{minipage}
\end{figure}
\begin{figure}[H]
  \begin{minipage}[h]{0.32\linewidth}
    \centering
\includegraphics[width=43mm]{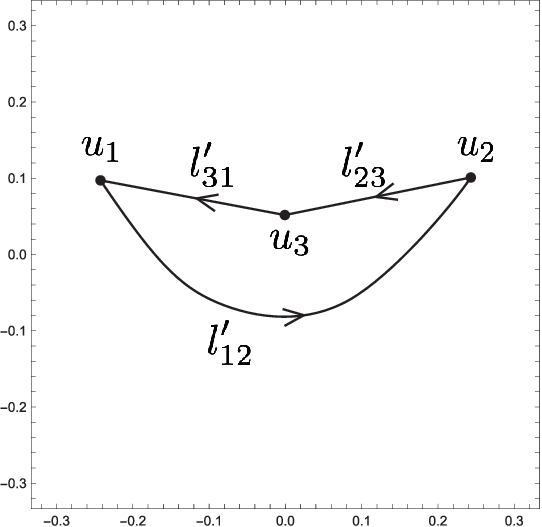}
    \caption{$x^{(7)}=(0.15+0.37i,\frac{1+i}{2})$}\label{6.7}
  \end{minipage}
  \begin{minipage}[h]{0.32\linewidth}
    \centering
\includegraphics[width=43mm]{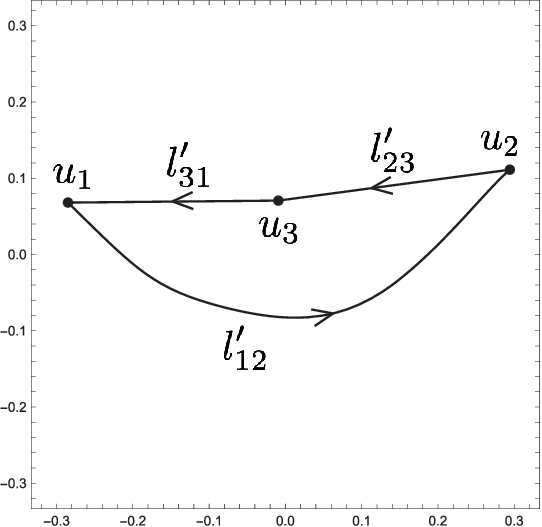}
    \caption{$x^{(8)}=(0.15+0.45i,\frac{1+i}{2})$}\label{6.8}
  \end{minipage}
    \begin{minipage}[h]{0.32\linewidth}
    \centering
\includegraphics[width=43mm]{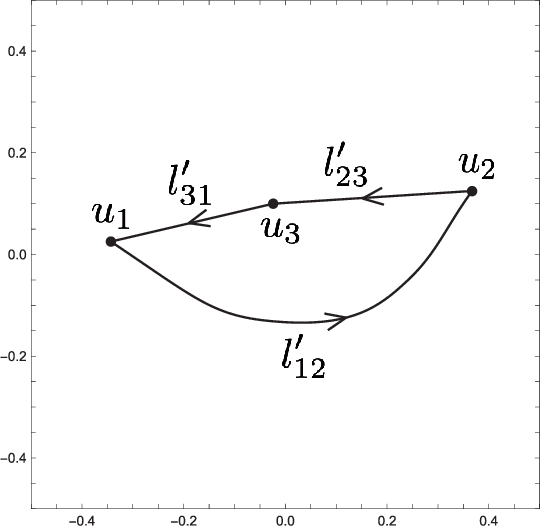}
    \caption{$x^{(9)}=(0.15+0.56i,\frac{1+i}{2})$}\label{6.9}
  \end{minipage}
\end{figure}
\begin{figure}[H]
  \begin{minipage}[h]{0.32\linewidth}
    \centering
\includegraphics[width=43mm]{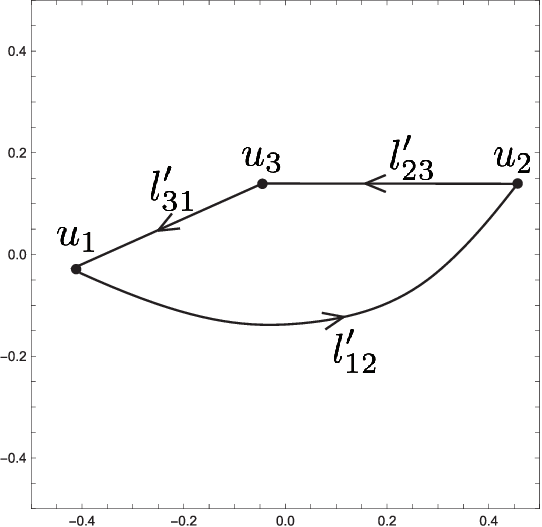}
    \caption{$x^{(10)}=(0.15+0.69i,\frac{1+i}{2})$}\label{6.10}
  \end{minipage}
  \begin{minipage}[h]{0.32\linewidth}
    \centering
\includegraphics[width=43mm]{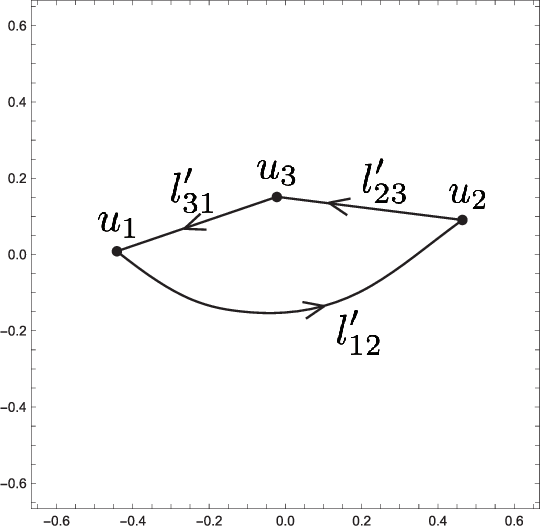}
    \caption{$x^{(11)}=(0.22+0.69i,\frac{1+i}{2})$}\label{6.11}
  \end{minipage}
    \begin{minipage}[h]{0.32\linewidth}
    \centering
\includegraphics[width=43mm]{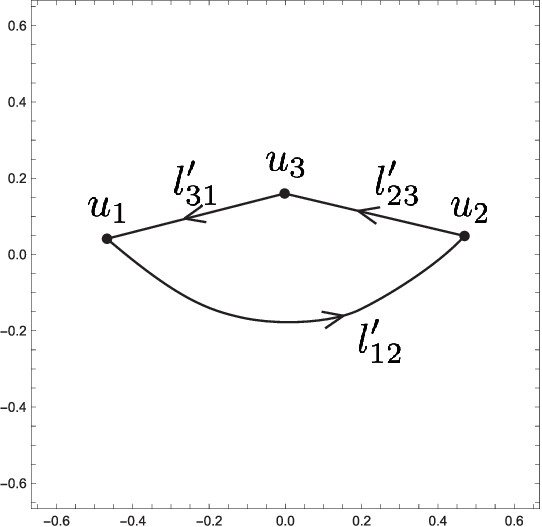}
    \caption{$x^{(12)}=(0.28+0.69i,\frac{1+i}{2})$}\label{6.12}
  \end{minipage}
\end{figure}
Figures \ref{6.13}-\ref{6.15} show the sections of the Stokes set for $x_{2}=0, (2+i)/4, (1+i)/2$, respectively.
Larger dots designate the corresponding sections of the turning point set. 
The values of $x^{(j)}$ ($j=1,2,\dots,12$)
are found in the captions of Figures \ref{6.1}-\ref{6.12} and $x^{(13)}=(0.45+0.69i,0.5+0.5i)$. We write $x_{1}^{(j)}$ the 
$x_{1}$-coordinate of $x^{(j)}$. The smaller dots in Figure \ref{6.15} show the
location of $x_{1}^{(j)}$  for $j=5,6,\dots,13$. 
Those in Figures \ref{6.13} and \ref{6.14} designate $x^{(1)}=(x_1,x_2)=(0.15,0)$ and $x^{(4)}=(x_1,x_2)=(0.15,0.5+0.25i)$, respectively. 
As mentioned earlier, Figures 1 and 4 show the locations in the \(y\)-plane of the singularities of the Borel transform of the WKB solutions at \(x^{(1)}\) and \(x^{(4)}\).
\begin{figure}[H]
  \begin{minipage}[h]{0.48\linewidth}
    \centering
\includegraphics[width=65mm]{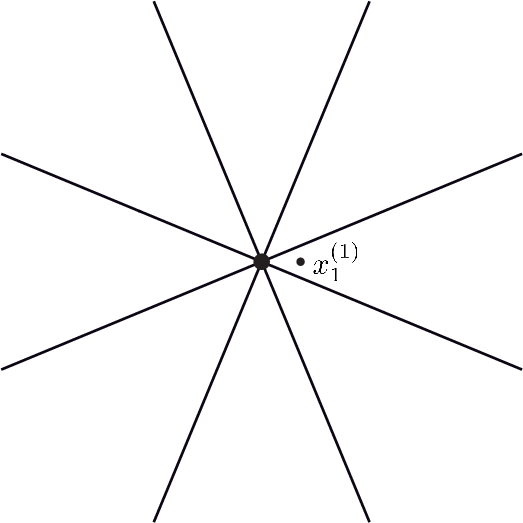}
    \caption{$x_{2}=0$}\label{6.13}
  \end{minipage}
  \begin{minipage}[h]{0.48\linewidth}
    \centering
\includegraphics[width=65mm]{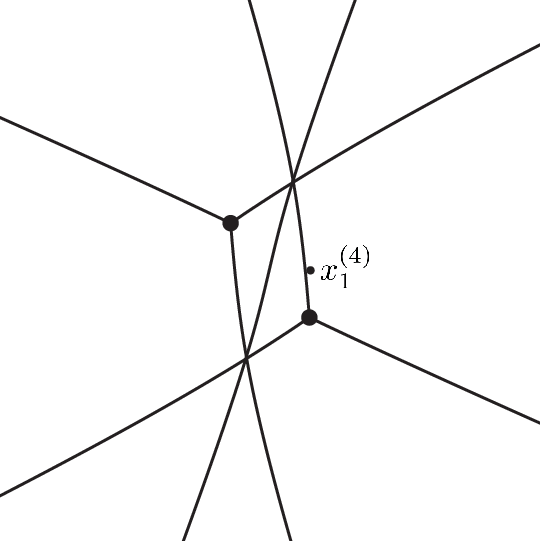}
    \caption{$x_{2}=0.5+0.25i$}\label{6.14}
  \end{minipage}
\end{figure}

\begin{figure}[H]
  \begin{minipage}[h]{0.48\linewidth}
    \centering
\includegraphics[width=65mm]{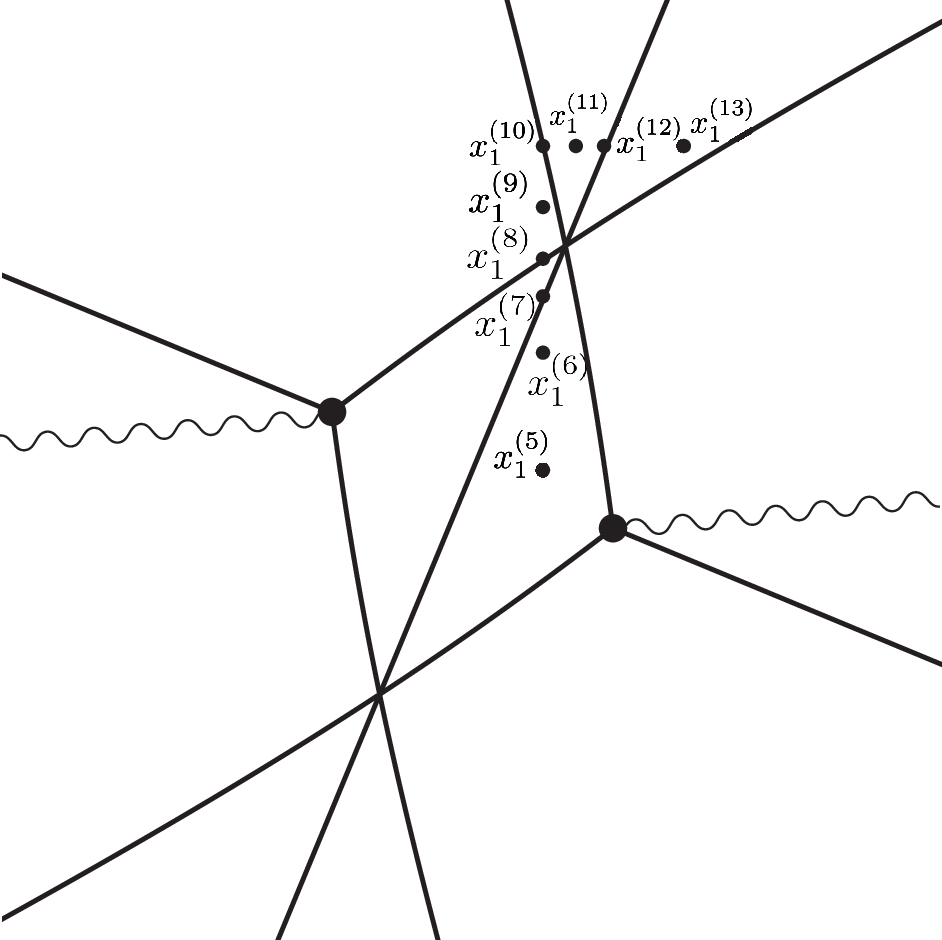}
    \caption{$x_{2}=0.5+0.5i$}\label{6.15}
      \end{minipage}
        \begin{minipage}[h]{0.48\linewidth}
            \centering
\includegraphics[width=58mm]{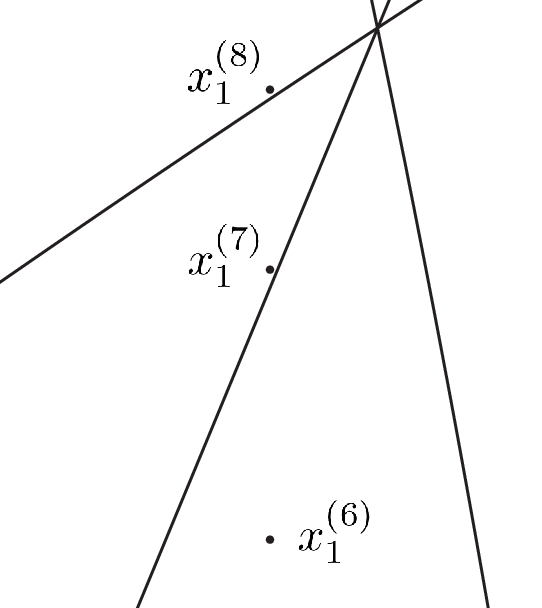}
    \caption{~}\label{6.16}
  \end{minipage}
\end{figure}

In Figure \ref{6.15}, the points $x_1^{(j)}$ $(j = 5, 6, \ldots, 12)$ are $x_1$-coordinates of $x_1^{(j)}$ designated in the caption of Figure $j$.
 If ${\rm Im}(u_{j}-u_{k})=0$ for some $j\neq k$, then
$x$ is contained in the Stokes set. 
In the case of \(x^{(4)}\), since \({\rm Im}(u_{2} - u_{3})\) is close to 0, it can be seen from Figure \ref{6.4} that \(x^{(4)}\) is near a point on the Stokes surface. 
Similar discussions can be applied to \(x = x^{(7)}, x^{(8)}, x^{(10)}, x^{(12)}\).
In particular, Figure \ref{6.16} is an enlarged view of a portion of Figure \ref{6.15}.
Hence there are 5 intersection points of the Stokes set and the path of continuation which are 
close to $x=x^{(4)}, x^{(7)}, x^{(8)}, x^{(10)}, x^{(12)}$,
respectively. 
As shown in Figures \ref{6.1}-\ref{6.3}, one observes that as $x$ varies from $x^{(1)}$ to $x^{(3)}$, the condition ${\rm Im}(u_j-u_k)\neq0$ holds for some $j\neq k$ on the polygonal line from $x^{(1)}$ to $x^{(3)}$.
This implies that $x^{(1)}, x^{(2)}, x^{(3)}$ are contained in the same Stokes region,
which is denoted by ${\mathcal D}_{1}$. 
Let ${\mathcal D}_{2}$, ${\mathcal D}_{3}$,  ${\mathcal D}_{4}$, ${\mathcal D}_{5}$ and ${\mathcal D}_{6}$ designate the
Stokes regions containing $x^{(5)}$, $x^{(7)}$, $x^{(9)}$, $x^{(11)}$ and $x^{(13)}$, respectively. 
We write $\Psi_{\ell}^{k}$ the Borel sum of $\psi_{\ell}$ in ${\mathcal D}_{k}$ ($k=1, 2, \dots, 6; \ell=1, 2, 3$).
In Figure \ref{6.17}, the sections of ${\mathcal D}_{k}$ ($k=1, 2, \dots, 6$) by the hyperplane $x_{2}=1/2+i/2$ are shown by the same letters.
\begin{figure}[H]
    \centering
\includegraphics[width=65mm]{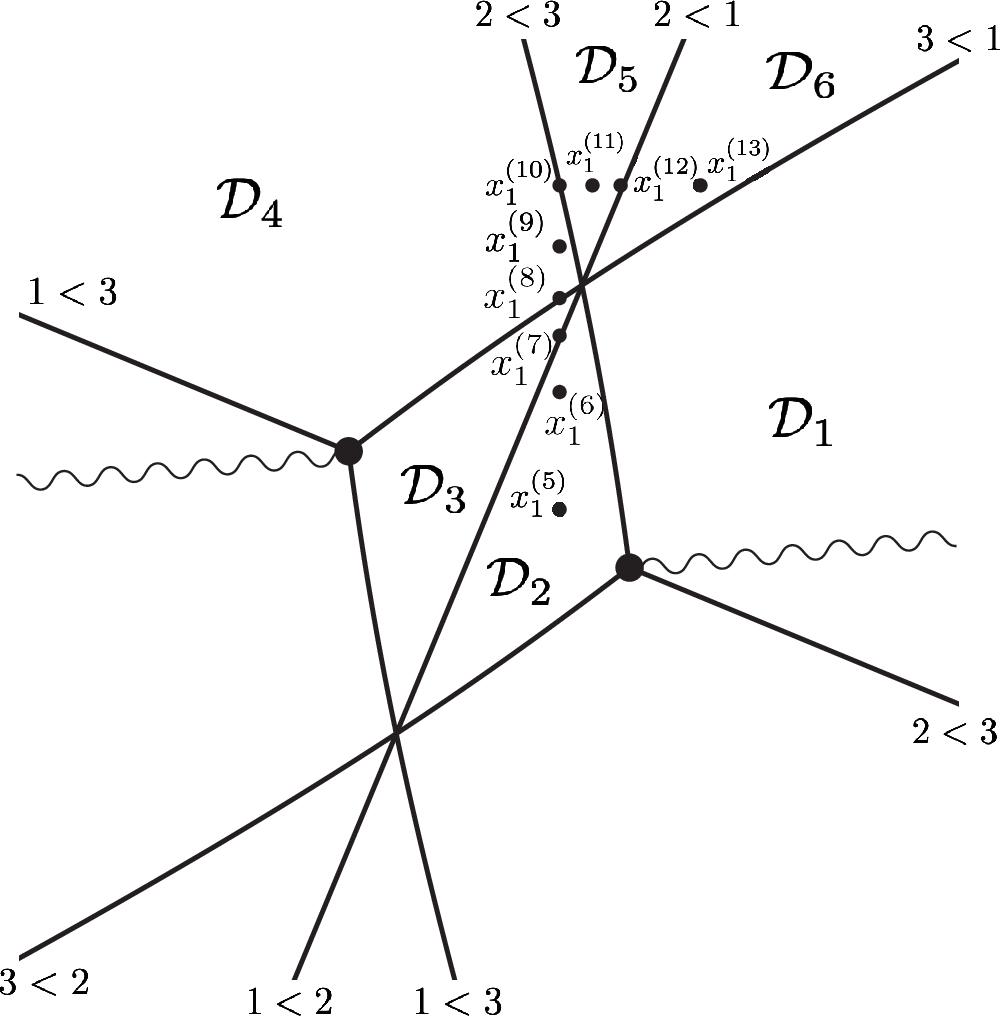}
    \caption{$x_{2}=0.5+0.5i$}\label{6.17}
\end{figure}
During the analytic continuation of $\psi_{\ell,B}$ (in $x$-variable) from
$x^{(1)}$ to $x^{(5)}$, $u_{1}$ never crosses the moving segment $u_{2}u_{3}$. 
As is remarked above, there is a crossing point of the segment $x^{(4)}x^{(5)}$ and
the Stokes surface which is close to $x^{(4)}$. On the point, $\psi_{3}$ dominates
$\psi_{2}$ in the sense that ${\rm Re}\, u_{3}<{\rm Re}\, u_{2}$ holds.
Modifying the path of integration of the definition of $\Psi_{3}^{1}$ \cite[Figure 2.3]{KT}, 
we see that $\Psi_{3}^{2}$ picks up the Laplace integral of 
$\Delta_{u_2} \psi_{3,B}$. 
Using Theorem \ref{thmdisp}, we have
\begin{equation*}
\Delta_{u_2} \psi_{3,B}=\Delta_{u_{2}, l}'\psi_{3,B}=-\psi_{2,B}\quad (l=l_{32}').
\end{equation*}
When $x$ moves from $x_{1}^{(4)}$ to $x_1^{(5)}$, we observe that
$\Im(u_1-u_3)\neq 0$, $\Im(u_1-u_2) \neq 0$ and there are no singular points on $L_2=\{u_2+r\mid r>0\}$. 
This implies there are no Stokes phenomena for $\psi_{1}, \psi_{2}$ between ${\mathcal D}_{1}$ and ${\mathcal D}_{2}$.
Hence we obtain
\begin{equation*}
\left\{
\begin{split}
\Psi_{1}^{1}&=\Psi_{1}^{2},\\
\Psi_{2}^{1}&=\Psi_{2}^{2},\\
\Psi_{3}^{1}&=\Psi_{3}^{2}-\Psi_{2}^{2}.
\end{split}
\right.
\end{equation*}
There is a crossing point of the segment $x^{(6)}x^{(7)}$ and the Stokes surface which is close to $x^{(7)}$.
At the point, we have ${\rm Im}(u_{1}-u_{2})=0$ and ${\rm Re}\, u_{1}<{\rm Re}\, u_{2}$. 
The Stokes curve containing the point obtained as the section of the Stokes surface at $x_2=0.5+0.5 i$ is labeled $2<1$ in Figure \ref{6.17}. Other Stokes curves in the Figure are labeled in a similar way. (The opposite side of the Stokes curve is labeled $1<2$ because the origin is a virtual turning point in the sense of \cite{AKT,HKT}.)

Thus $\psi_{1}$
might have a Stokes phenomenon there. But by using Theorem \ref{thmdisp}, we have
\[
\Delta_{u_{2}}\psi_{1,B}=\Delta_{u_{2},l}'\psi_{1,B}=0\quad (l=\hat{l}_{12}').
\]
Since $\Im(u_3-u_1)\neq 0$, $\Im(u_3-u_2) \neq 0$ and there are no singular points on $L_2$ when $x$ moves from $x_{1}^{(6)}$ to $x_1^{(7)}$, there are no Stokes phenomena for $\psi_{2}, \psi_{3}$ between ${\mathcal D}_{2}$ and ${\mathcal D}_{3}$. Hence we have
 \begin{equation*}
\left\{
\begin{split}
\Psi_{1}^{2}&=\Psi_{1}^{3},\\
\Psi_{2}^{2}&=\Psi_{2}^{3},\\
\Psi_{3}^{2}&=\Psi_{3}^{3}.
\end{split}
\right.
\end{equation*}
This means that there is no Stokes phenomenon for $\psi_{\ell}$ 
between ${\mathcal D}_{2}$ and ${\mathcal D}_{3}$ ($\ell=1,2,3$).
Near $x^{(8)}$, there is a crossing point of the segment $x^{(7)}x^{(8)}$ and the Stokes surface.
On that point, ${\rm Im}(u_{1}-u_{3})=0$ and ${\rm Re}\,u_{1}<{\rm Re}\,u_{3}$, that is,
$\psi_{1}$ dominates $\psi_{3}$ there. Theorem \ref{thmdisp} implies
\[
\Delta_{u_{3}}\psi_{1,B}=\Delta_{u_{3},l}'\psi_{1,B}=-\psi_{3,B}\quad (l=l_{13}').
\]
Hence we have
\begin{equation*}
\left\{
\begin{split}
\Psi_{1}^{3}&=\Psi_{1}^{4}-\Psi_{3}^{4},\\
\Psi_{2}^{3}&=\Psi_{2}^{4},\\
\Psi_{3}^{3}&=\Psi_{3}^{4}.
\end{split}
\right.
\end{equation*}
Similar discussions as above
show
\begin{equation*}
\left\{
\begin{split}
\Psi_{1}^{4}&=\Psi_{1}^{5},\\
\Psi_{2}^{4}&=\Psi_{2}^{5},\\
\Psi_{3}^{4}&=\Psi_{3}^{5}+\Psi_{2}^{5}
\end{split}
\right.
\end{equation*}
and
\begin{equation*}
\left\{
\begin{split}
\Psi_{1}^{5}&=\Psi_{1}^{6}+\Psi_{2}^{6},\\
\Psi_{2}^{5}&=\Psi_{2}^{6},\\
\Psi_{3}^{5}&=\Psi_{3}^{6}.
\end{split}
\right.
\end{equation*}

These connection formulas for $\psi_{\ell}$ are essentially obtained by Hirose \cite{HiroseM,Hirose}.
He uses the local theory around a simple turning point of a higher-order ordinary differential 
equation with the large parameter (\cite{AKKoT}). On the other hand, our discussion is totally straightforward and explicit. 

\subsection*{Acknowledgements}
The authors are grateful to Professors S. Hirose, S. Izumi and Y. Takei for helpful comments and discussions.
Furthermore, they would like to thank the referee for his careful reading and helpful comments, which improved this article.
The first and second authors are supported by JSPS KAKENHI grant number 18K03385.
The third author is supported by Foundation of Research Fellows, The Mathematical Society of Japan and supported by JSPS KAKENHI grant number 26K17004.

\end{document}